\documentclass{article}

\usepackage[noadjust]{cite}
\usepackage{amsmath,amssymb,amsfonts}
\usepackage{algorithmic}
\usepackage{graphicx}
\usepackage{algorithm,algorithmic}
\usepackage{textcomp}
\usepackage{stmaryrd}
\usepackage{url}
\usepackage{makecell} 
\usepackage{nccmath,mathtools}
\usepackage{diagbox}
\usepackage{booktabs}
\usepackage{multirow}
\usepackage{threeparttable}

\usepackage{enumitem}

\usepackage{amsthm}

\usepackage[margin=1.1in]{geometry}

\newcolumntype{C}[1]{>{\centering\let\newline\\\arraybackslash\hspace{0pt}}m{#1}}

\theoremstyle{plain}
\newtheorem{theorem}{Theorem}
\newtheorem{lemma}{Lemma}

\theoremstyle{definition}
\newtheorem{assumption}{Assumption}

\theoremstyle{remark}
\newtheorem{remark}{Remark}

\DeclarePairedDelimiterX\Econd[2]{[}{]}{#1 \mkern2mu\delimsize\vert\mkern2mu #2}

\newcommand{\ext}[1]{\mkern#1mu}

\begin{document}
\title{Variance-Reduced Gradient Estimator for Nonconvex Zeroth-Order Distributed Optimization}
\author{Huaiyi Mu, Yujie Tang, Jie Song, and Zhongkui Li%
\thanks{The authors are with the School of Advanced Manufacturing and Robotics, Peking University, Beijing, China (e-mail: {\tt\small huaiyi.mu@stu.pku.edu.cn, yujietang@pku.edu.cn, jie.song@pku.edu.cn, zhongkli@pku.edu.cn}).}}
\date{}

\maketitle

\begin{abstract}
This paper investigates distributed zeroth-order optimization for smooth nonconvex problems, targeting the trade-off between convergence rate and sampling cost per zeroth-order gradient estimation in current algorithms that use either the $2$-point or $2d$-point gradient estimators. 
We propose a novel variance-reduced gradient estimator that either randomly renovates a single orthogonal direction of the true gradient or calculates the gradient estimation across all dimensions for variance correction, based on a Bernoulli distribution.
Integrating this estimator with gradient tracking mechanism allows us to address the trade-off. We show that the oracle complexity of our proposed algorithm is upper bounded by $\mathcal{O}(d/\epsilon)$ for smooth nonconvex functions and by $\mathcal{O}(d\kappa\ln (1/\epsilon))$ for smooth and gradient dominated nonconvex functions, where $d$ denotes the problem dimension and $\kappa$ is the condition number. Numerical simulations comparing our algorithm with existing methods confirm the effectiveness and efficiency of the proposed gradient estimator.
\end{abstract}

\section{Introduction}
\label{sec:introduction}
We consider a multi-agent system with $N$ agents, where the agents are connected by a communication network that allows them to exchange information for decentralized decision-making. The goal of this group of agents is to collaboratively minimize the global objective function
\begin{equation} \label{distributed_optimization}
    \begin{aligned}
         f(x) \coloneqq \frac{1}{N}\sum_{i=1}^N f_i(x),
    \end{aligned}
\end{equation}
i.e., to solve $\min_{x\in\mathbb{R}^d}f(x)$, in a decentralized fashion. Here $x\in\mathbb{R}^d$ is the global decision variable. Each function $f_i: \mathbb{R}^d \to \mathbb{R}$ represents the local objective function for agent~$i$, known only to the agent itself; $f_i$ is assumed to be smooth but may be nonconvex. We also impose the restriction that each agent may only use zeroth-order information of $f_i$ during the optimization procedure.

Decentralized optimization has gained considerable interest due to its broad applications in areas such as multi-agent system coordination~\cite{8695072}, power systems \cite{8006269}, communication networks \cite{10093046}, etc. For smooth and convex objective functions, the decentralized gradient descent (DGD) algorithm achieved a convergence rate of $\mathcal{O}(\frac{\log t}{\sqrt{t}})$ with decreasing step-sizes~\cite{nedic2009distributed, yuan2016convergence}. To improve efficiency, gradient tracking (GT) methods~\cite{shi2015extra, qu2017harnessing, nedic2017achieving} employ a fixed step-size, attaining a sublinear convergence rate of $\mathcal{O}(\frac{1}{t})$, comparable to centralized gradient descent method. Under the assumption of strong convexity on the objective functions, DGD can achieve a convergence rate of $\mathcal{O}(\frac{1}{t})$ as shown in \cite{pu2021sharp,8786146}, while GT achieves a linear convergence rate of $\mathcal{O}(\lambda^k)$ as shown in \cite{qu2017harnessing, 7963560, 8988200}. In many real-world applications, the objective functions can be nonconvex, making distributed nonconvex optimization critical for applications in machine learning \cite{nedic2020distributed}, sensor networks \cite{ren2021distributed}, and robotics control \cite{carnevale2024nonconvex}. For smooth nonconvex functions, DGD achieves convergence to a stationary point with a rate of $\mathcal{O}(\frac{1}{\sqrt{t}})$~\cite{3295285, zeng2018nonconvex}, while various GT methods can achieve convergence to a stationary point with a rate of $\mathcal{O}(\frac{1}{t})$ \cite{scutari2019distributed, sun2019convergence}, comparable to the results obtained in the convex case~\cite{shi2015extra}.
For distributed non-convex optimization on time-varying communication networks, \cite{tatarenko2017non} employed the perturbed push-sum method to achieve a rate of $\mathcal{O}(\frac{1}{t})$. Reference~\cite{sun2019distributed} derived lower rate bounds for distributed non-convex first-order optimization, and developed algorithms embedding the polynomial filtering techniques that can match the lower bounds. The paper~\cite{bai2024complexity} considered distributed smooth nonconvex finite-sum optimization under the Polyak--Łojasiewicz condition, achieving a linear convergence rate.

The aforementioned optimization algorithms rely on first-order information. However, in some scenarios, the gradient is unavailable or is costly to obtain, and only zeroth-order information is accessible, such as in optimization with black-box models \cite{li2021surrogate}, optimization with bandit feedback \cite{bubeck2012regret}, fine-tuning language models \cite{malladi2023fine}, etc. To address this issue, various gradient-free optimization methods have been proposed. Particularly, algorithms based on zeroth-order gradient estimators have attracted considerable attention recently due to their flexibility and scalability. 
For centralized gradient-free optimization, \cite{zhang2024boosting} proposed a one-point estimator with residual feedback for centralized online optimization. The paper~\cite{chen2025regression} introduced a regression-based single-point gradient estimator for centralized zeroth-order optimization. 
The works \cite{nesterov2017random,duchi2015optimal} investigated the 2-point zeroth-order gradient estimator for centralized problems, which produces a biased stochastic gradient by using the function values of two randomly sampled points. 
In terms of distributed zeroth-order optimization, \cite{mhanna2024zero} investigated one-point gradient estimators for distributed stochastic optimization under the convex setting. For strongly convex problems, \cite{sahu2018communication} employed a 2-point zeroth-order gradient estimator for stochastic decentralized gradient descent algorithm and achieves sublinear convergence. In \cite{hajinezhad2019zone}, a 2-point stochastic zeroth-order oracle was integrated with the method of multipliers for distributed zeroth-order optimization under various network topologies.
The paper~\cite{yi2022zeroth} combined 2-point gradient estimator with the primal-dual method, achieving linear speedup under the gradient dominance assumption on non-smooth objective functions. 
The works~\cite{lin2024decentralized,sahinoglu2024online} proposed gradient-free methods for decentralized non-smooth non-convex
optimization using 2-point gradient estimator.
In~\cite{kiefer1952stochastic}, the $2d$-point gradient estimator was proposed, where $d$ is the dimension of the state variable for each agent. The $2d$-point estimator provides more precise gradient estimates than the 2-point estimator, at the expense of higher computational complexity per construction. 
The work~\cite{9199106} combined the 2-point gradient estimator with DGD and the $2d$-point gradient estimator with GT for nonconvex multi-agent optimization, which lead to convergence rates that are comparable with their first-order counterparts. However, \cite{9199106} also argued that there seems to be a trade-off between the \emph{convergence rate} and the \emph{sampling cost per zeroth-order gradient estimation}, when one attempts to combine zeroth-order gradient estimation techniques with different distributed optimization frameworks. This trade-off arises from the inherent variance of the 2-point estimator in distributed settings and the high sampling burden of the $2d$-point estimator.

To overcome this trade-off, we aim to design a variance-reduced zeroth-order gradient estimator with a scalable sampling number of function values that is independent of the dimension $d$. Variance reduction techniques have been extensively utilized in machine learning \cite{guo2021machine} and stochastic optimization \cite{wang2013variance}. In \cite{johnson2013accelerating}, variance reduction was employed in centralized stochastic gradient descent with strongly convex objectives, achieving a linear convergence rate. 
Reference~\cite{fang2018spider} proposes the SPIDER variance reduction method for stochastic nonconvex optimization, and \cite{li2021page} introduced the PAGE variance reduction framework that employs probabilistic update for the reference gradient. \cite{liu2018zeroth} applied a 2-point gradient estimator and used variance reduction for zeroth-order nonconvex centralized stochastic optimization, achieving sublinear convergence. In \cite{kazemi2024efficient}, variance-reduced zeroth-order gradient estimator was employed for solving non-smooth composite optimization problems. Note that these works only focused on centralized problems. For decentralized finite-sum minimization, variance reduction has been used to accelerate convergence, as seen in~\cite{xin2020variance, jiang2022distributed}. 
In these works, the variance reduction techniques were employed for reducing the variance caused by the finite-sum structure but not for reducing the inherent variance of the 2-point zeroth-order gradient estimators. The work~\cite{chen2023zeroth} utilizes a 2-point gradient estimator together with variance reduction for decentralized nonconvex optimization in the stochastic setting, assuming bounded dissimilarity between local objectives.

\begin{table*}[htbp]
\label{table:comparison}
    \centering
    \begin{threeparttable}
    \caption{Oracle Complexities and Sampling Costs per Iteration of Algorithms for Deterministic Nonconvex Optimization}
    \small
    \renewcommand{\arraystretch}{1.8}
   \begin{tabular}{C{2cm}C{3cm}C{3.5cm}C{2.5cm}C{2.5cm}}
        \toprule
        &
        &
        \makecell[c]{smooth \\ nonconvex} &
        \makecell[c]{gradient\\ dominated} &
        \makecell[c]{sampling cost\\ per iteration} \\
        \midrule

        \multirow[c]{2.75}{2cm}{\makecell[c]{distributed \\ first-order}}
        & DGD
        & $\mathcal{O}(1/\epsilon^2)$~\cite{3295285}
        & $\mathcal{O}(\kappa^2/\epsilon)$~\cite{pu2021sharp} (strongly convex)
        & -- \\
        \cline{2-5}
        & gradient tracking
        & $\mathcal{O}(1/\epsilon)$~\cite{scutari2019distributed}
        & $\mathcal{O}(\kappa \ln(1/\epsilon))$~\cite{xu2021distributed} (strongly convex)
        & -- \\
        \midrule
        \multirow[c]{3}{2cm}{\makecell[c]{centralized \\ zeroth-order}}
        & \makecell[c]{\cite{nesterov2017random}}
        & \makecell[c]{$\mathcal{O}(d/\epsilon)$}
        & \makecell[c]{$\mathcal{O}(d\kappa \ln(1/\epsilon))$}  (strongly convex)
        & \makecell[c]{$\Theta(1)$}  \\
        \cline{2-5}
        & SPIDER-SZO~\cite{fang2018spider}
        & $\mathcal{O}(d/\epsilon)$
        & --
        & $\Theta(1)$ \\
        \cline{2-5}
        & SPIDER-Coord~\cite{ji2019improved}
        & $\mathcal{O}(d/\epsilon)$
        & $\mathcal{O}(d\kappa^2 \ln(1/\epsilon))$
        & $\Theta(d)$  \\
        \midrule

        \multirow[c]{2.5}{2cm}{\makecell[c]{distributed\\ nonsmooth\\ zeroth-order}}
        & $\text{DGFM}^+$~\cite{lin2024decentralized}
        & $\mathcal{O}(d^{\frac{3}{2}}/(\delta\epsilon^{\frac{3}{2}}))$ \quad (nonsmooth, stochastic)
        & \makecell[c]{--}
        & \makecell[c]{$\Theta(1/\sqrt{\epsilon})$}   \\
        \cline{2-5}
        & ME-DOL~\cite{sahinoglu2024online}
        & $\mathcal{O}(d/(\delta\epsilon^{\frac{3}{2}}))$ \ \ \ \quad (nonsmooth, stochastic)
        & --
        & $\Theta(1)$ \\
        \midrule

        \multirow[c]{6}{2cm}{\makecell[c]{distributed\\ zeroth-order}}
        & DGD-2p~\cite{9199106}
        & $\mathcal{\tilde{O}}(d/\epsilon^2)$
        & $\mathcal{O}(d\kappa^2 /\epsilon)$
        & $\Theta(1)$ \\
        \cline{2-5}
        & GT-2d~\cite{9199106}
        & $\mathcal{O}(d/\epsilon)$
        & $\mathcal{O}(d\kappa^{\frac{4}{3}} \ln(1/\epsilon))$
        & $\Theta(d)$ \\
        \cline{2-5}
        & ZONE~\cite{hajinezhad2019zone}
        & $\mathcal{O}(\gamma(d)/\epsilon^2)$
        & --
        & $\Theta(1/\epsilon)$ \\
        \cline{2-5}
        & DZO primal-dual~\cite{yi2021linear}
        & $\mathcal{O}(d/\epsilon)$
        & $\mathcal{O}(d\kappa \ln(1/\epsilon))$
        & $\Theta(d)$  \\
        \cline{2-5}
        & \textbf{our algorithm}
        & $\mathcal{O}(d/\epsilon)$
        & $\mathcal{O}(d\kappa \ln(1/\epsilon))$
        & $\Theta(1)$ \\
        \bottomrule
    \end{tabular}  
    \begin{tablenotes}[online]
    \item[1)] The listed oracle complexities are the number of zeroth-order queries needed to obtain a point $x$ satisfying $\mathbb{E}[\|\nabla f(x)\|^2]\leq\epsilon$ for the smooth nonconvex case and $\mathbb{E}[f(x)-f^\ast]\leq\epsilon$ for the gradient dominated case, respectively.
    \item[2)] The rates provided in \cite{hajinezhad2019zone} do not include explicit dependence on $d$; we use $\gamma(d)$ to denote this dependence.
    \item[3)] For distributed nonsmooth nonconvex optimization, the oracle complexity is the number of zeroth-order queries needed to obtain a point $x$ satisfying $\min\{\|g\|^2:g\in\partial_\delta f(x)\}\leq\epsilon$, where $\partial_\delta f(x)$ denotes the Goldstein $\delta$-subdifferential; see~\cite{lin2024decentralized,sahinoglu2024online} for precise definitions.
    \item[4)] The notation $\mathcal{\tilde{O}}$ omits logarithmic factors in $d$ and/or $\epsilon$.
    \end{tablenotes}
    \end{threeparttable}
\end{table*}

In this paper, we propose a new distributed zeroth-order optimization method that integrates variance reduction techniques with the gradient tracking framework, to address the trade-off between \emph{convergence rate} and \emph{sampling cost per zeroth-order gradient estimation} in existing distributed zeroth-order algorithms under the deterministic nonconvex optimization setting. Specifically, We leverage the variance reduction (VR) mechanism to design a novel variance-reduced gradient estimator for distributed nonconvex zeroth-order optimization problems, as formulated in \eqref{distributed_optimization}. We then combine this new zeroth-order gradient estimation method with the gradient tracking framework, and the resulting algorithm is able to achieve both fast convergence and low sampling cost per zeroth-order gradient estimation. We also provide rigorous convergence analysis of our proposed algorithm under the smoothness assumption as well as the gradient-dominance assumption. 
The derived oracle complexities match the state-of-the-art dependence on the dimension $d$.
To the best of the authors' knowledge, this is the first work that attempts to address the aforementioned trade-off for zeroth-order distributed optimization  with deterministic objectives for both the general nonconvex and the gradient-dominated cases.
We refer to Table~\ref{table:comparison} for a comparison of the oracle complexities and sampling costs per iteration between our algorithm and some related existing algorithms.
Numerical experiments demonstrate that our proposed algorithm enjoys superior convergence speed and accuracy compared to existing zeroth-order distributed optimization algorithms \cite{9199106, hajinezhad2019zone}, reaching lower optimization error with the same number of samples.

This article is an extension of our preliminary work in a conference submission~\cite{publications_tang}. 
Inspired by the PAGE method~\cite{li2021page}, we have redesigned our variance-reduced gradient estimator, leading to a complexity bound $\mathcal{{O}}(d/\epsilon)$ that has improved dependence on the dimension $d$. We also expand our analysis to gradient-dominated smooth nonconvex functions and derive a superior complexity bound $\mathcal{O}(d\ln(1/\epsilon))$. The Appendices of this journal version provides the complete proofs of all the theorems and critical lemmas.

\vspace{3pt}
\textbf{Notations:} The set of positive integers up to $m$ is denoted as $[m] = \{1,2,\cdots, m \}$. The $i$-th component of a vector $x$ is denoted as $[x]_i$. The spectral norm and spectral radius of a matrix $A$ are represented by $\sigma(A)$ and $\rho(A)$, respectively. For a vector $x\in\mathbb{R}^d$, $\|x\|$ refers to the $\ell_2$ Euclidean norm. For a matrix $A$, $\|A\|_2$ represents the spectral norm induced by $\|\cdot\|$. For two matrices $M$ and $N$, $M\otimes N$ denotes the Kronecker product. We denote $\mathbb{B}_d$ as the closed unit ball in $\mathbb{R}^d$, and $\mathbb{S}_{d-1} = \{ x\in\mathbb{R}^d : \|x\|=1 \}$ as the unit sphere. $\mathcal{U}(\cdot)$ denotes the uniform distribution.

\section{Formulation And Preliminaries}
\subsection{Problem Formulation}
We consider a network consisting of $N$ agents connected via an undirected communication network. The topology of the network is represented by the graph $\mathcal{G} = (\mathcal{N}, \mathcal{E})$, where $\mathcal{N}$ and $\mathcal{E}$ represent the set of agents and communication links, respectively.
The distributed consensus optimization problem~\eqref{distributed_optimization} can be equivalently reformulated as follows:
\begin{equation}\label{distributed_optimization_2}
\begin{aligned}
\min_{x_1,\ldots,x_N\in\mathbb{R}^{d}}\ \ & \frac{1}{N}\sum_{i=1}^{N} f_i(x_i) \\
\text{s.t.} \ \ \ \ & x_1=x_2=\cdots=x_N,
\end{aligned}
\end{equation}
where $x_i\in\mathbb{R}^d$ now represents the local decision variable of agent $i$, and the constraint $x_1=\cdots=x_N$ requires the agents to achieve global consensus for the final decision. During the optimization procedure, each agent may obtain other agents' information only via exchanging messages with their neighbors in the communication network. We further impose the restriction that only zeroth-order information of the local objective function is available to each agent. In other words, in each iteration, agent $i$ can query the function values of $f_i$ at finitely many points.

The following assumptions will be employed later in this paper.
\begin{assumption} \label{assumption_smooth_f^*}
Each $f_i : \mathbb{R}^d \to \mathbb{R}$ is $L$-smooth, i.e., we have 
\begin{equation} 		\|\nabla f_i(x) - \nabla f_i(y)\| \leq L\| x-y \| \label{L_smooth_assump}
\end{equation}
for all $x,y\in \mathbb{R}^d$ and $i=1,\ldots,N$. Furthermore, $f^* = \inf_{x\in\mathbb{R}^d}f(x)>-\infty$.
\end{assumption}

\begin{assumption} \label{assumption_smooth+mu}
Each $f_i : \mathbb{R}^d \to \mathbb{R}$ is $L$-smooth, and the global objective function $f : \mathbb{R}^d \to \mathbb{R}$ is $\mu$-gradient dominated, i.e., we have
\begin{subequations}
\begin{align}
\|\nabla f_i(x) - \nabla f_i(y)\| & \leq L\| x-y \|, \\
\|\nabla f(x)\|^2 & \geq 2\mu (f(x)-f^*) \label{mu_gradient}
\end{align}
\end{subequations}
for any $x,y\in \mathbb{R}^d$ and $i=1,\ldots,N$, where $f^\ast \coloneqq \inf_{x\in\mathbb{R}^d} f(x)> -\infty$.
\end{assumption}

The condition~\eqref{mu_gradient} is also know as the Polyak-Łojasiewicz inequality~\cite{polyak1964gradient, lojasiewicz1963topological}. With the gradient-dominance condition, alongside the smoothness assumption, non-convex optimization has the potential to achieve linear convergence~\cite{yi2022zeroth}.

\subsection{Preliminaries on Distributed Zeroth-Order Optimization}
When gradient information of the objective function is unavailable, one may construct gradient estimators by sampling the function values at a finite number of points, which has been shown to be a very effective approach by existing literature. We first briefly introduce two types of gradient estimators \cite{9199106} that are commonly used in noiseless distributed optimization.

Let $h:\mathbb{R}^d\rightarrow\mathbb{R}$ be a continuously differentiable function. One version of the 2-point zeroth-order gradient estimator for $\nabla h(x)$ has the following form:
\begin{equation} \label{2_point}
	\begin{aligned}
		& G_h^{(2)}(x,u,z) = d \cdot \frac{h(x+uz)-h(x-uz)}{2u}z, \\
	\end{aligned}
\end{equation}
where $u$ is a positive scalar called the \emph{smoothing radius} and $z$ is a random vector sampled from the distribution $\mathcal{U}(\mathbb{S}_{d-1})$. One can show that the expectation of the 2-point gradient estimator is the gradient of a smoothed version of the original function \cite{nemirovskij1983problem,1070486}, i.e.,
\[
\mathbb{E}_{z \thicksim \mathcal{U}(\mathbb{S}_{d-1})}[G_h^{(2)}(x,u,z)] = \nabla h^{u}(x),
\]
where $h^{u}(x) = \mathbb{E}_{y\thicksim \mathcal{U}(\mathbb{B}_d)}[h(x+uy)]$. As the smoothing radius $u$ tends to zero, the expectation of the 2-point gradient estimator approaches to the true gradient $\nabla h(x)$. 

By combining the simple 2-point gradient estimator~\eqref{2_point} with the decentralized gradient descent framework, one obtains the following algorithm for distributed zeroth-order consensus optimization~\eqref{distributed_optimization_2}:
\begin{equation}
x_i^{k+1} = \sum_{j=1}^N W_{ij}\!\left(x^k_j - \eta_k\, G_{f_j}^{(2)}(x^k_j, u^k, z^k_j)\right),
\end{equation}
which we shall call DGD-2p in this paper. Here $x^k_i$ denotes the local decision variable of agent $i$ at the $k$-th iteration, $W\in\mathbb{R}^{N\times N}$ is a weight matrix that is taken to be doubly stochastic, and $\eta_k$ is the step-size at iteration $k$. Since each construction of the 2-point gradient estimator \eqref{2_point} requires sampling only two function values, we can see that DGD-2p can achieves low sampling cost per zeroth-order gradient estimation. However, as shown by \cite{9199106}, DGD-2p achieves a relatively slow convergence rate $\mathcal{O}(\sqrt{d/m}\log m)$, where $m$ denotes the number of function value queries. \cite{9199106} argued that this slow convergence rate is mainly due to the inherent variance of the 2-point gradient estimator, bounded by
\[
\mathbb{E}_{z \thicksim \mathcal{U}(\mathbb{S}_{d-1})}\!\left[\| G_h^{(2)}(x, u, z) \|^2\right] \lesssim d\|\nabla h(x)\|^2 + u^2L^2d^2
\]
under the assumption that function $h$ is $L$-smooth. In a distributed optimization algorithm, each agent's local gradient $ \nabla f_i(x^k_i)$ does not vanish to zero even if the system has reached consensus and optimality. Consequently, the inherent variance of 2-point gradient estimator is inevitable and will considerably slow down the convergence rate.

To achieve higher accuracy for zeroth-order gradient estimation, existing literature has also proposed the $2d$-point gradient estimator:
\begin{equation} \label{2d_point}
\begin{aligned}
G_h^{(2d)}(x,u) = \sum_{l=1}^d \frac{h(x+ue_l) - h(x-ue_l)}{2u}e_l.
\end{aligned}
\end{equation}
Here $e_l \in \mathbb{R}^d$ is the $l$-th standard basis vector such that $[e_l]_j = 1$ when $j = l$ and $[e_l]_j = 0$ otherwise. It can be shown that $\|G_h^{(2d)}(x,u) -\nabla h(x) \| \leq \frac{1}{2}u L\sqrt{d}$ when $h$ is $L$-smooth (see, e.g., \cite{9199106}). Consequently, if we assume the function values of $h$ can be obtained accurately and machine precision issues in numerical computations are ignored, then the $2d$-point gradient estimator can achieve arbitrarily high accuracy when approximating the true gradient. By combining~\eqref{2d_point} with the distributed gradient tracking method, one obtains the following algorithm: 
\begin{equation} \label{GT_2d}
\begin{aligned}
x^{k+1}_i ={} & \sum_{j=1}^N W_{ij}\!\left(
x^k_j - \eta\, s^k_j
\right), \\
s^{k+1}_i ={} & \sum_{j=1}^N W_{ij}\!\left( s^k_j \!+\!  G_{f_j}^{(2d)}(x^{k+1}_j, u^{k+1}) \!-\! G_{f_j}^{(2d)}(x^k_j, u^k) \right),
    \end{aligned}
\end{equation}
which we shall call GT-$2d$. Here the auxiliary state variable $s^k_j$ in \eqref{GT_2d} tracks the global gradient across iterations. Distributed zeroth-order optimization algorithms that utilize the $2d$-point gradient estimator, such as GT-$2d$, can achieve faster convergence due to more precise estimation of the true gradients that allows further incorporation of gradient tracking techniques. However, GT-$2d$ has higher sampling cost per gradient estimation compared to DGD-2p: As shown in \eqref{2d_point}, $2d$ points need to be sampled for each construction of the gradient estimator. This high sampling cost may lead to poor scalability when the dimension $d$ is large.

We remark that the $2d$-point gradient estimator \eqref{2d_point} can also be interpreted as the expectation of the following \emph{coordinate-wise} gradient estimator:
\begin{equation} \label{directional_gradient}
    \begin{aligned}
         G_h^{(c)}(x,u,l) = d \cdot \frac{h(x + u e_l) - h(x - u e_l)}{2u} e_l,\ \  l\in [d],
    \end{aligned}
\end{equation}
and we have
\begin{equation} \label{expectation_2d}
\begin{aligned}
G_h^{(2d)}(x,u) = \mathbb{E}_{l\sim\mathcal{U}[d]}\!\left[G_h^{(c)}(x,u,l)\right],
\end{aligned}
\end{equation}
where $\mathcal{U}[d]$ denotes the discrete uniform distribution over the set $\{1,\ldots,d\}$. The coordinate-wise gradient estimator in \eqref{directional_gradient} shares a similar structure with the 2-point gradient estimator in \eqref{2_point}. The key difference is that in \eqref{directional_gradient}, we restrict the perturbation direction $e_l$ to lie in the $d$ orthogonal directions associated with the standard basis, instead of uniformly sampled from the unit sphere.

\section{Our Algorithm}
To address the trade-off between convergence rate and sampling cost per gradient estimation in zeroth-order distributed optimization, we employ a variance reduction mechanism \cite{li2021page} to design an improved gradient estimator. The intuition is to combine the best of both worlds, i.e., the precise approximation feature of the $2d$-point gradient estimator and the low-sampling feature of the 2-point gradient estimator.

Let $k$ denote the iteration number. For each agent, We use a random variable $\zeta_i^k $ generated from the Bernoulli distribution $ \mathrm{Ber}(p)$ as an activation indicator for updating the gradient estimation of the agents. We then propose the variance-reduced gradient estimator (VR-GE) $g_i^k$ as follows:

\begin{equation} \label{algorithm_i}
    \begin{aligned}
            g_i^k \!=\!\! \begin{cases}
        G_{f_i}^{(2d)}(x_i^k, u_i^k),  & \zeta_i^k \!=\! 1, \\
        g_i^{k-\!1} \!+\! G_{f_i}^{(c)}(x_i^k, u_i^k, l_i^k) \!-\! G_{f_i}^{(c)}(x_i^{k-\!1}, u_i^{k-\!1}, l_i^k),  & \zeta_i^k \!=\! 0.
    \end{cases}
    \end{aligned}
\end{equation}
When $\zeta_i^k = 1$, agent $i$ updates $g_i^k$ using the $2d$-point gradient estimator. This ensures an accurate gradient estimation during iteration $k$ at the cost of $2d$ sample points. When $\zeta_i^k = 0$, agent $i$ randomly selects one orthogonal direction $l_i^k$ from all dimensions, i.e., $l_i^k \sim\mathcal{U}[d]$. The agent then constructs the coordinate-wise gradient estimators $G_{f_i}^{(c)}(x_i^k , u_i^k , l_i^k)$ and $G_{f_i}^{(c)}(x_i^{k-1}, u_i^{k-1}, l_i^k)$ using the values of state variables from two consecutive iterations, $x^k_i$ and $x_i^{k-1}$. Subsequently, it updates $g_i^k$ using the prior information from $g_i^{k-1}$ as a basis and renovates the $l_i^k$th component of $g_i^k$ by employing the variation in coordinate-wise gradient estimator. This enables gradient updating along the direction $l_i^k$ using only 4 sampling points.

It is not hard to show that the local gradient estimator $g_i^k$ can track the true gradient $\nabla f_i(x_i^k)$ in expectation with high accuracy. Indeed, given the randomness of $\zeta_i^k$ and $l_i^k$, we can derive that
\begin{equation} \label{expectation}
    \begin{aligned}
        \mathbb{E}_{\zeta_i^k ,l_i^k}\Econd*{g_i^k}{x^k_i,x_i^{k-1}}
          ={}  & pG_{f_i}^{(2d)}(x^k_i,u^k_i) + (1-p)\Big(\mathbb{E}_{\zeta_i^k ,l_i^k}\Econd*{g_i^{k-1}}{x^k_i,x_i^{k-1}}  \\
          & \qquad +G_{f_i}^{(2d)}(x^k_i,u^k_i)  - G_{f_i}^{(2d)}(x_i^{k-1},u_i^{k-1})\Big), 
    \end{aligned}
\end{equation}
where we have used~\eqref{expectation_2d} in the equality. Taking the total expectation and applying mathematical induction, it is straightforward to derive $\mathbb{E}\!\left[g_i^k\right]=\mathbb{E}\big[G_{f_i}^{(2d)}\!(x^k_i,\!u^k_i)\big]$. Considering that $\big\|G_{f_i}^{(2d)}(x^k_i,u^k_i) -\nabla f_i(x^k_i)\big\|\leq \frac{1}{2}u^k_i L\sqrt{d}$ when $f_i$ is $L$-smooth, we then obtain $\big\|\mathbb{E}\!\left[g_i^k\right]-\mathbb{E}\!\left[\nabla f_i(x_i^k)\right]\big\|\leq \frac{1}{2}u^k_i L\sqrt{d}$. By selecting a sufficiently small smoothing radius $u_i^k$, the expectations of the gradient estimator $g_i^k$ and the true gradient will be aligned.

The expected number of function value samples required per construction of VR-GE is $4+(2d-4)p$. For $d \geq 3$, by choosing $p = \frac{C}{2d-4}$ for some positive constant $C$, this becomes $4+C$ which is independent of the dimension $d$. This gives VR-GE the potential to decrease the sampling cost in high-dimensional zeroth-order distributed optimization by appropriately adjusting the probability $p$. In the following section, specifically in Lemma~\ref{g^*-nabla_f}, we will rigorously analyze the variance of VR-GE and demonstrate its variance reduction property.

In designing our distributed zeroth-order optimization algorithm, we further leverage the gradient tracking mechanisms. Existing literature (including \cite{shi2015extra,qu2017harnessing,nedic2017achieving}, etc.) has demonstrated that gradient tracking mechanisms help mitigate the gap in the convergence rates between distributed optimization and centralized optimization when the objective function is smooth. Drawing inspiration from this advantage, we incorporate the variance-reduced gradient estimator with gradient tracking mechanism to design our algorithm.

The details of the proposed algorithm are outlined in Algorithm~\ref{main_algorithm}. Here $\alpha>0$ is the step-size; Steps 1 and 5 implement the gradient tracking mechanism, while Steps 2--4 implement our proposed variance-reduced gradient estimator~\eqref{algorithm_i}. The convergence guarantees of Algorithm~\ref{main_algorithm} will be provided and discussed in the next section.

\begin{algorithm}[tbh]
\caption{Distributed Zeroth-Order Optimization Algorithm with Variance Reduced Gradient Tracking Estimator}\label{main_algorithm}
 \begin{algorithmic}
 \STATE { Initialization : $ x_i^0 = \mathbf{0}_d , s_i^0 = g_i^0 = G_{f_i}^{(2d)}(x_i^0, u_i^0) $. } 
  \item {\textbf{for} $k=0,1,2,\cdots$ \textbf{do}}
  \item {\quad \textbf{for each} $i\in[N]$ \textbf{do}}
  \item {\quad 1. Update $x_i^{k+1}$ by \[x_i^{k+1} = \sum_{j=1}^N W_{ij} (x_j^k - \alpha s_j^k).\]}
  \item {\quad 2. Select $l_i^{k+1}$ uniformly at random from $[d]$. }
  \item {\quad 3. Generate $\zeta_i^{k+1}\sim \mathrm{Ber}(p).$}
  \item {
  \quad 4. Construct the VR-GE $g_{i}^{k+1}$ by: \\
  \quad\phantom{4. }
  If $\zeta_i^{k+1} = 1$, compute
  \[ g_i^{k+1} = G_{f_i}^{(2d)}(x_i^{k+1}, u_i^{k+1}); \] \\ 
  \quad\phantom{4. }
  If $\zeta_i^{k+1} = 0$, compute 
  \[
  \begin{aligned}
    g_i^{k+1} ={} & g_i^k + G_{f_i}^{(c)}(x_i^{k+1}, u_i^{k+1}, l_i^{k+1})- G_{f_i}^{(c)}(x_i^k, u_i^k, l_i^{k+1}).
  \end{aligned} \]
  }
  \item {\quad 5. Update $s^{k+1}_i$ by
  \[
  \begin{aligned}
  s_i^{k+1} ={} & \sum_{j=1}^N W_{ij} (s_j^k + g_j^{k+1} - g_j^k).
  \end{aligned}
  \]
  }
  \item {\quad \textbf{end}}
  \item {\textbf{end}}
 \end{algorithmic}
\end{algorithm}

\section{Convergence Results} \label{main_results}
In this section, we present the convergence results of Algorithm 1 under Assumption 1 and Assumption 2, respectively. We provide proof outlines of Theorem~\ref{theorem1} and Theorem~\ref{theorem2} in Section~\ref{convergence_analysis}, while detailed proofs of critical lemmas are postponed to the Appendices.

For the subsequent analysis, we denote
\[
x^k = \begin{bmatrix}
x^k_1 \\ \vdots \\ x^k_N
\end{bmatrix},\ 
s^k = \begin{bmatrix}
s^k_1 \\ \vdots \\ s^k_N
\end{bmatrix},\ 
g^k = \begin{bmatrix}
g_1^k \\ \vdots \\ g_N^k
\end{bmatrix},\ 
\nabla F(x^k) = \begin{bmatrix}
\nabla f_1(x_1^k) \\ \vdots \\
\nabla f_N(x_N^k)
\end{bmatrix},
\]
and define the following quantities:
\begin{equation*}
\begin{aligned}
\delta^k &=\mathbb{E}\!\left[ f(\bar{x}^k) \right] -f^*, &
E_x^k &= \mathbb{E}\!\left[ \big\| x^k - \mathbf{1}_N \otimes \bar{x}^k \big\|^2 \right], \\
E_s^k &=\mathbb{E}\!\left[ \big\| s^k - \mathbf{1}_N \otimes \bar{g}^k \big\|^2 \right], &
E_g^k &=  \mathbb{E}\!\left[ \big\| g^k - \nabla F(x^k) \big\|^2 \right],
\end{aligned}
\end{equation*}
where $\bar{x}^k = \frac{1}{N}\sum_{i=1}^N x_i^k$ and $ \bar{g}^k = \frac{1}{N}\sum_{i=1}^N g_i^k $. Here, $\delta^k$ quantifies the optimality gap in terms of the objective value, $E_s^k$ and $E_g^k$ characterize the tracking errors, and $E_x^k$ characterizes the consensus error. We also denote $\sigma = \big\| W - \frac{1}{N}\mathbf{1}_N\mathbf{1}_N^T \big\|_2$. Furthermore, we introduce the following auxiliary quantities:
\[
C_u=d\!\left[(1-p)\!\left(4d+\frac{2}{p}\right)+\frac{p}{4}\right],
\ \ \chi = \frac{1}{4} - \frac{1}{8}\sqrt{3+\sigma^2}.
\]
It can be checked that $\chi\in\big(\frac{1-\sigma^2}{32},\frac{1-\sigma^2}{29}\big)$.

\begin{theorem} \label{theorem1}
Under Assumption~\ref{assumption_smooth_f^*}, suppose the parameters of Algorithm~\ref{main_algorithm} satisfy the following conditions: i) $p\in (0 ,1]$; ii) $\sum_{\tau=0}^{\infty} (u_i^{\tau})^2 < \infty$ for all $i$; iii) $u_i^k$ is non-increasing; iv) the step-size is given by
\begin{align*}
& \alpha L = c\sqrt{\frac{p}{d(1-p)+1}},
\end{align*}
where $c$ is a positive constant bounded by $c\leq\big(\frac{1-\sigma^2}{28}\big)^2$. Denote
\[
R_0 = \frac{1}{N}\!\left[(E_g^0)^2+\frac{48c^2p}{1\!-\!\sigma^2}(E_s^0)^2\right]^{\frac{1}{2}},
R_u=\frac{C_u}{Np}\sum_{\tau=0}^\infty\sum_{i=1}^N (u_i^\tau)^2.
\]
Then we have
\begin{equation}
\label{eq:smooth_stationarity_rate}
		  \frac{1}{k}\sum_{\tau=0}^{k-1}  \mathbb{E}\!\left[ \big\| \nabla f(\bar{x}^{\tau}) \big\|^2 \right]
		  \leq \frac{1}{k} \!\left( \frac{2}{\alpha}\delta^0 +  \frac{4R_0}{\chi p}  + \frac{9L^2 R_u}{2\chi } \right)\!,
\end{equation}
\begin{equation}
\label{eq:smooth_x_consensus_rate}
		  \frac{1}{kN}\sum_{\tau=0}^{k-1}E_x^k
		  \leq \frac{1}{k} \Bigg(\frac{216c^2}{\alpha L^2\chi} \delta^0 + \frac{3R_0}{\chi pL^2}  +  \frac{5R_u}{ 2\chi } \Bigg),
\end{equation}
and
\begin{equation}
\begin{aligned}
		  \frac{1}{kN}\sum_{\tau=0}^{k-1}\sum_{i=1}^N  \mathbb{E} \!\left[ \left\|s_i^{\tau} - \nabla f(\bar{x}^{\tau}) \right\|^2 \right]
		  \leq{} & \frac{1}{k} \Bigg( \frac{108c^2}{\alpha^2L\chi} \delta^0 + \frac{3R_0}{2\alpha L\chi p}  + \frac{5LR_u}{4\alpha \chi}  \Bigg),
\end{aligned}
\end{equation}
\end{theorem}

Theorem~\ref{theorem1} shows that the convergence rate of Algorithm 1 under Assumption 1 is $\mathcal{O}(\frac{1}{k})$, which aligns with the rate achieved for distributed nonconvex optimization with gradient tracking using first-order information~\cite{scutari2019distributed}. In addition, each iteration of VR-GE requires $4 + (2d-4)p$ function value queries on average. As long as $p<1$, the averaged sampling number for VR-GE is less than that for the $2d$-point estimator.

We next provide some discussions on the query complexities of Algorithm 1 under different choices of~$p$.

1) Assuming $p = 1/d$ (or $p=\gamma/d$ for some numerical constant $\gamma>0$), the sampling cost per iteration on average is $\Theta(1)$ and the step-size $\alpha$ is $\mathcal{O}(1/d)$ in terms of dependence on dimension $d$. By choosing the smoothing radii to satisfy $\sum_{\tau=0}^\infty (u_i^\tau)^2\propto d^{-2}$, the convergence rate~\eqref{eq:smooth_stationarity_rate} becomes $\mathcal{O}( d/m )$ with respect to the number of function value queries $m$, which can be justified by simple algebraic calculation. One can also obtain oracle complexity result for Algorithm~\ref{main_algorithm} from this convergence rate result: Under Assumption~\ref{assumption_smooth_f^*}, given an arbitrary $\epsilon>0$, the number of zeroth-order queries per agent needed to achieve $\frac{1}{k}\sum_{\tau=0}^{k-1}\mathbb{E}[\|\nabla f(\bar x^\tau)\|^2]\leq\epsilon$ can be upper bounded by $\mathcal{O}(d/\epsilon)$.

2) When $p = 1$, Algorithm 1 reduces to GT-$2d$~\cite{9199106}. In this case, the sampling cost per iteration is $\Theta(d)$, and the step-size $\alpha$ required by Theorem~\ref{theorem1} is $\mathcal{O}(1)$ in terms of dependence on dimension $d$. As a result, the rate given by~\eqref{eq:smooth_stationarity_rate} will reduce to the existing result $\mathcal{O}( d/m )$ given in~\cite{9199106}.

    We point out that the complexity bound of $\mathcal{O}\big(d/\epsilon\big)$ for Algorithm~\ref{main_algorithm} is as favorable as the complexity bound for GT-$2d$ in~\cite{9199106} and DZO in \cite{yi2021linear} in terms of the dependence on $\epsilon$ and the problem dimension $d$.
    This indicates that our algorithm achieves the state-of-the-art complexity result concerning its dependence on $\epsilon$ and $d$, while maintaining a constant expected number of samples per iteration by suitably choosing the probability $p$, regardless of the size of the problem dimension. This approach can potentially decrease the execution time for a high-dimensional distributed optimization algorithm under limited resources. As demonstrated in the simulation section, Algorithm 1 converges faster than GT-$2d$ and DZO and achieves higher accuracy than DGD-2p with the same number of samples (i.e., zeroth-order queries).

Next, we show the convergence result under the Polyak-Łojasiewicz condition in addition to smoothness.
\begin{theorem} \label{theorem2}
Under Assumption~\ref{assumption_smooth+mu}, suppose the parameters of Algorithm~\ref{main_algorithm} satisfy the same conditions as in Theorem~\ref{theorem1}.
Then we have 
\[
\begin{aligned}
    \mathbb{E}\!\left[f(\bar{x}^k)\right]-f^* 
     \leq{} \mathcal{O}(\lambda^k) +  \frac{9\alpha L^2}{2\chi } \mathfrak{R}_u^k ,
\end{aligned}
\]
and
\begin{align*}
     \frac{1}{N} \sum_{i=1}^N  \mathbb{E}\!\left[ \left\|x_i^k - \bar{x}^k \right\|^2 \right] 
     \leq{} &  \mathcal{O}(\lambda^k) + \frac{9}{8}p \mathfrak{R}_u^k , \\
     \frac{1}{N} \!\sum_{i=1}^N  \mathbb{E} \!\left[ \left\|s_i^{k} - \nabla f(\bar{x}^k) \right\|^2 \right] 
     \leq{} & \mathcal{O}(\lambda^k) + \frac{9 Lp}{16\alpha } \mathfrak{R}_u^k,
\end{align*}
where
\begin{align*}
\lambda &= \max \!\left\{ 1-\alpha\mu, 1-\frac{1}{2}\chi p \right\},\\ 
\mathfrak{R}_u^k &= \frac{C_u}{pN}\sum_{\tau=0}^{k-1}\lambda^{\tau}\sum_{i=1}^N\big(u_i^{k-\tau-1}\big)^2.
\end{align*}
\end{theorem}

From Theorem~\ref{theorem2}, we can further establish the oracle complexity of our algorithm when the objective functions are smooth and gradient-dominated. Specifically, when one chooses $p\propto \frac{1}{d}$, we have $\alpha\propto 1/d$ and $1-\lambda=\Theta(1/d)$. In addition,  by choosing $u_i^k$ to be sufficiently small, $\frac{9\alpha L^2}{2\chi}\mathfrak{R}_u^k$ will be dominated by the first term $\mathcal{O}(\lambda^k)$. Therefore, to achieve $\mathbb{E}[f(\bar{x}^k)]-f^\ast\leq\epsilon$, the number of zeroth-order queries per agent needed to achieve $\mathbb{E}\!\left[f(\bar{x}^k)\right]-f^\ast\leq\epsilon$ can be upper bounded by $\mathcal{O}\big(\frac{1}{1-\lambda}\ln (1/\epsilon) \big)
=\mathcal{O}(d\kappa\ln (1/\epsilon))$, where $\kappa=L/\mu$ is the condition number of the problem. 
We also point out that the oracle complexity $\mathcal{O}(d\kappa\ln (1/\epsilon))$ is consistent with the state-of-the-art result regarding its dependence on $\epsilon$, $d$ and~$\kappa$.

\section{Theoretical Analysis} \label{convergence_analysis}
In this section, we provide the theoretical proofs for the convergence and complexity performance of Algorithm~\ref{main_algorithm}, as outlined by the theorems in Section~\ref{main_results}.

\subsection{Bounding the Variance of VR-GE}
The variance of VR-GE is essential for convergence proof of Algorithm \ref{main_algorithm} and we provide analysis details in this subsection. We first rewrite Algorithm \ref{main_algorithm} as follows:
\begin{subequations} \label{algorithm_compact}
	\begin{align}
		& x^{k+1} = (W \otimes I_d)(x^k -\alpha s^k), \label{algorithm_compact_1} \\
		& s^{k+1} = (W \otimes I_d)(s^k + g^{k+1} - g^k).\label{algorithm_compact_2}
	\end{align}
\end{subequations}

We now derive a bound on the expected difference between variance-reduced gradient estimator and the true gradient in the following lemma.

\begin{lemma} \label{g^*-nabla_f}
	Suppose each $f_i : \mathbb{R}^d \to \mathbb{R}$ is $L$-smooth. Let $g_i^k$ be generated by~\eqref{algorithm_i}. Then it holds that
	\begin{equation} \label{variance}
	\begin{aligned}
			\mathbb{E}\!\left[ \left\| g_i^{k+1} - \nabla f_i(x_i^{k+1}) \right\|^2 \right] 
			\leq{} & (1-p)\!\left(1+\frac{p}{2}\right)\mathbb{E}\!\left[ \left\| g_i^k - \nabla f_i(x_i^k) \right\|^2 \right] + C_uL^2(u_i^k)^2 \\
            & + 6d(1-p)L^2\,\mathbb{E}\!\left[\left\| x_i^{k+1} - x_i^k \right\|^2\right],	
	\end{aligned}
	\end{equation}
    where $C_u=d((1-p)(4d+2p^{-1})+p/4)$.
\end{lemma}

\begin{proof}
See Appendix~\ref{appendix:proof_lemma_g^*-nabla_f}.
\end{proof}

\begin{remark}
The bound~\eqref{variance} demonstrates a contraction factor $(1-p)(1+p/2)=1-p/2-p^2/2$ for the estimation error of VR-GE across successive iterations. Consequently, as Algorithm \ref{main_algorithm} approaches consensus and optimum and the smoothing radius approaches zero, the estimation error between the VR-GE and the true gradient diminishes. Thus, VR-GE offers reduced variance compared to the 2-point gradient estimator while requiring fewer samples than the $2d$-point gradient estimator on average. 
\end{remark}

\subsection{Proof of Theorem 1}
The proof relies on four lemmas. The first lemma analyzes the evolution of function value $f(\bar{x}^k)$ by exploiting the $L$-smoothness property. 

\begin{lemma} \label{error_optimization}
	Under Assumption~\ref{assumption_smooth_f^*}, we have
\begin{equation}\label{fbarTemp}
\mathbb{E}\!\left[\big\|\nabla f(\bar{x}^k)-\bar{g}^k\big\|^2\right]
\leq\frac{2}{N}E_g^k+\frac{2L^2}{N}E_x^k,
\end{equation}
and
 \begin{align}
     		\delta^{k+1} \leq{} & \delta^k - \frac{\alpha}{2}\mathbb{E}\!\left[ \big\| \nabla f(\bar{x}^k) \big\|^2 \right] + \frac{\alpha}{N} E_g^k + \frac{\alpha L^2}{N}E_x^k \nonumber\\
            & - \left( \frac{1}{2\alpha} \!-\! \frac{L}{2} \right) \mathbb{E}\!\left[\left\| \bar{x}^{k+1} - \bar{x}^k \right\|^2\right]. 
       \label{delta_k}
 \end{align}
\end{lemma}

\begin{proof}
See Appendix~\ref{appendix:proof_lemma_error_optimization}.
\end{proof}

Lemma~\ref{error_optimization} derives a bound for the optimization error $\delta^k$. We need to further bound the consensus error $E_x^k$, alongside the tracking errors $E_s^k$ and $E_g^k$. This is tackled by the following lemma.

\begin{lemma} \label{LMI_lemma}
	Suppose we choose $p\in (0 ,1] $ and $\alpha L = c\sqrt{\frac{p}{d(1-p) + 1}} $, where $c$ is a positive constant bounded by $c\leq \big(\frac{1-\sigma^2}{28}\big)^2$. Then we have the following component-wise inequality:
\begin{equation} \label{linear_matrix}
	v^{k+1} \leq A v^k + b^k,
\end{equation}
where $v^k = \! \left[ E_x^{k} , E_g^{k} , E_s^k \right]^T$, and
\begin{align*}
	    A ={} & \! \begin{bmatrix}
            \mfrac{1+2\sigma^2}{3}   &      0      &   \mfrac{3\alpha^2}{1-\sigma^2}  \\[3pt]
            48 d(1\!-\!p)L^2       &   (1\!-\!p)\big(1\!+\!\mfrac{p}{2}\big)     & 24 d(1\!-\!p)L^2\alpha^2   \\[3pt]
            \mfrac{96( 3d(1\!-\!p) \!+\! 1 )L^2}{1-\sigma^2} & \mfrac{18}{1-\sigma^2}   &  \mfrac{2+\sigma^2}{3} 
    \end{bmatrix}\!,\\ 
	b^k ={} & \!\begin{bmatrix}
		0 \\[3pt]
		{ 24 Nd(1 \!-\! p)L^2 } \mathbb{E}[\| \bar{x}^{k+1} - \bar{x}^k \|^2] + L^2C_u\sum_i(u_i^k)^2 \\[3pt]
        \ext{-2}
		\mfrac{48( 3d(1\!-\!p) \!+\! 1 )NL^2}{1\!-\!\sigma^2}\mathbb{E}[ \| \bar{x}^{k+1} \!-\! \bar{x}^k \|^2] \!+\! \mfrac{6L^2C_u\sum_i(u_i^k)^2}{1\!-\!\sigma^2}
        \ext{-2}
	\end{bmatrix}\!\!.
\end{align*}
\end{lemma}

\begin{proof}
See Appendix~\ref{appendix:proof_lemma_LMI_lemma}.
\end{proof}

Next, we derive a bound on the accumulated consensus error and tracking errors over iterations using Lemma~\ref{LMI_lemma}.
\begin{lemma} \label{Sigma_inequality}
	Suppose we choose $p\in(0 ,1] $ and $\alpha L = c\sqrt{\frac{p}{d(1-p) + 1}} $, where $c$ is a positive constant bounded by $c\leq \big(\frac{1-\sigma^2}{28}\big)^2$. We denote 
    \[
    \begin{aligned}
    E_c^k &= E_x^{k} + \frac{18\alpha^2}{(1-\sigma^2)^2} E_s^{k}, \\
    E_f^k &= \left[\frac{4(3d(1-p)+1)L^2(1-\sigma^2)^3}{81\alpha^2} (E_c^k)^2 + (E_g^k)^2\right]^{\frac{1}{2}}.
    \end{aligned}
    \]
    Then we have
\begin{equation} \label{EfLemmaForm}
    \begin{aligned}
        E_f^{k+1} \leq{} &  9( 3d(1\!-\!p) \!+\! 1 ) NL^2\mathbb{E}\!\left[ \left\| \bar{x}^{k+1} - \bar{x}^k \right\|^2 \right] \\
        & + (1-\chi p) E_f^k + \frac{9L^2C_u\sum_i(u_i^k)^2}{8},
    \end{aligned}
\end{equation}
where $\chi = \frac{1}{4} - \frac{1}{8}\sqrt{3+\sigma^2}$. Furthermore,
\begin{equation} \label{sumEf}
    \begin{aligned}
        \sum_{\tau=0}^{k} E_f^{\tau} \leq{} & 
        \frac{9(3d(1\!-\!p) \!+\! 1 ) NL^2}{\chi p} \sum_{m=0}^{k-1} \mathbb{E}\!\left[\left\| \bar{x}^{\tau+1} - \bar{x}^\tau \right\|^2\right] \\
        & + \frac{1}{\chi p}E_f^0 + \frac{9L^2C_u}{ 8\chi p}\sum_{\tau=0}^{k-1} \sum\nolimits_i(u_i^\tau)^2.
    \end{aligned}
\end{equation}
\end{lemma}

\begin{proof}
See Appendix~\ref{appendix:proof_lemma_Sigma_inequality}.
\end{proof}

The inequality~\eqref{sumEf} will be applied in analyzing the convergence of $\bar{x}^k$ to stationarity, while the inequality~\eqref{EfLemmaForm} will be used to analyze the consensus error. Following this, we derive a bound for $\mathbb{E}[\| \bar{x}^{k+1} - \bar{x}^k \|^2]$, which will subsequently be used in the analysis of the consensus error.
\begin{lemma} \label{barxkkbarxk}
    Suppose we choose $p\in(0 ,1] $ and $\alpha L = c\sqrt{\frac{p}{d(1-p) + 1}} $, where $c$ is a positive constant bounded by $c\leq \big(\frac{1-\sigma^2}{28}\big)^2$. We have 
    \begin{equation} \label{barxkkLemmaForm}
        \begin{aligned}
            \mathbb{E}\!\left[\left\| \bar{x}^{k+1} - \bar{x}^k \right\|^2\right] \leq 2\alpha^2\mathbb{E}\!\left[\big\| \nabla f(\bar{x}^k) \big\|^2\right] + \frac{5\alpha^2}{N}E_f^k.
        \end{aligned}
    \end{equation}
\end{lemma}

\begin{proof}
See Appendix~\ref{appendix:proof_lemma_barxkkbarxk}.
\end{proof}

Based on Lemmas~\ref{error_optimization},~\ref{Sigma_inequality}, and~\ref{barxkkbarxk}, we are now ready to prove Theorem 1. Since $\delta^k\geq 0$ for all $k$, we derive from~\eqref{delta_k} that
\begin{equation}\label{deltaInitial}
\begin{aligned}
     0 \leq{} & \delta^0 - \frac{\alpha}{2} \sum_{\tau=0}^{k-1} \mathbb{E}\!\left[\big\| \nabla f(\bar{x}^{\tau}) \big\|^2\right] + \sum_{\tau=0}^{k-1} \left( \frac{\alpha}{N} E_g^{\tau} + \frac{\alpha L^2}{N}E_x^{\tau} \right) \\
        & - \left(\frac{1}{2\alpha} - \frac{L}{2}\right) \sum_{\tau=0}^{k-1} \mathbb{E}\!\left[\left\| \bar{x}^{\tau+1} - \bar{x}^{\tau} \right\|^2\right].  
\end{aligned}
\end{equation}
Using $\alpha L = c\sqrt{\frac{p}{d(1-p) + 1}}$ and $c\leq \big(\frac{1-\sigma^2}{28}\big)^2 $, we derive from the definitions of $E_c^k$ and $E_f^k$ in Lemma~\ref{Sigma_inequality} that $E_g^k \leq E_f^k$ and
\begin{equation}\label{ExEcComp}
    E_x^k \leq E_c^k \leq \!\left[\frac{81\alpha^2}{4(3d(1\!-\!p)\!+\!1)L^2(1\!-\!\sigma^2)^3}\right]^{\!\frac{1}{2}} \!E_f^k < \frac{E_f^k}{L^2}.
\end{equation}
Combining them with~\eqref{deltaInitial}, we have
\begin{equation}
    \begin{aligned}
         0 \leq{} & \delta^0 - \frac{\alpha}{2} \sum_{\tau=0}^{k-1} \mathbb{E}\!\left[\big\| \nabla f(\bar{x}^{\tau}) \big\|^2\right]  + \frac{2\alpha}{N} \sum_{\tau=0}^{k-1} E_f^{\tau} \\
        & - \left(\frac{1}{2\alpha} - \frac{L}{2}\right) \sum_{\tau=0}^{k-1} \mathbb{E}\!\left[\left\| \bar{x}^{\tau+1} - \bar{x}^{\tau} \right\|^2\right] \\
        \leq{} &  \delta^0 
        \!-\! \frac{\alpha}{2} \sum_{\tau=0}^{k-1} \mathbb{E}\!\left[\big\| \nabla f(\bar{x}^{\tau}) \big\|^2\right]
        \!+\! \frac{2\alpha}{\chi pN}E_f^0 + \frac{9\alpha L^2}{4\chi}R_u \\
        &\!-\! \left[\frac{1\!-\!\alpha L}{2\alpha}\!-\! \frac{ 18(3d(1\!-\!p) \!+\! 1 )\alpha L^2}{\chi p}\right]\!\sum_{\tau=0}^{k-1} \mathbb{E}\ext{-2}\Big[\ext{-2}\big\| \bar{x}^{\tau+1} \!-\! \bar{x}^{\tau} \big\|^2\Big]
    \end{aligned}
\end{equation}
where we have used \eqref{sumEf} and the definition of $R_u$ in the second inequality.
By $\alpha L = c\sqrt{\frac{p}{d(1-p) + 1}}$ with $c\leq \big(\frac{1-\sigma^2}{28}\big)^2 $, it is straightforward to verify that $\frac{1\!-\!\alpha L}{2\alpha}\!-\! \frac{ 18(3d(1\!-\!p) \!+\! 1 )\alpha L^2}{\chi p} > 0$. Consequently, we have 
 \begin{equation} \label{deltaIniMid}
         0 \leq{} \delta^0 \!-\! \frac{\alpha}{2} \sum_{\tau=0}^{k-1}\mathbb{E}\!\left[ \big\| \nabla f(\bar{x}^{\tau})\big\|^2\right] \!+\! \frac{2\alpha}{\chi pN} E_f^0 \!+\! \frac{9\alpha L^2}{4\chi}R_u,
\end{equation}
 We can then conclude from~\eqref{deltaIniMid} that
\begin{equation} \label{nablafMid}
    \begin{aligned}
     \frac{1}{k} \sum_{\tau=0}^{k-1} \mathbb{E}\!\left[\big\| \nabla f(\bar{x}^{\tau}) \big\|^2\right] 
        \leq \frac{1}{k} \left[ \frac{2}{\alpha}\delta^0 +  \frac{4}{\chi pN} E_f^0 + \frac{9L^2}{2\chi }R_u \right].
    \end{aligned}
\end{equation}
By using $\alpha L=c\sqrt{\frac{p}{d(1-p)+1}}$, it is straightforward to check that $E_f^0\leq NR_0$. The proof of~\eqref{eq:smooth_stationarity_rate} is now completed.

Next, we proceed to examine the consensus errors.
Plugging~\eqref{barxkkLemmaForm} into~\eqref{EfLemmaForm}, we obtain 
\begin{equation}
    \begin{aligned}
        E_f^{k+1}\! \leq{} & 9(3d(1\!-\!p) \!+\! 1 ) NL^2 \Big( 2 \alpha^2 \mathbb{E} \!\left[ \big\| \nabla f(\bar{x}^k) \big\|^2 \right] \!+\! \frac{5\alpha^2}{N} E_f^k \Big)  \\
        & + (1-\chi p) E_f^k  + \frac{9L^2C_u\sum_i(u_i^k)^2}{8}  \\
        \leq{} & \left( 1 - \frac{\chi p}{2} \right)E_f^k + 54Nc^2 p\mathbb{E} \!\left[ \big\| \nabla f(\bar{x}^k) \big\|^2 \right]  \\ & + \frac{9L^2C_u\sum_i(u_i^k)^2}{8},
    \end{aligned}
\end{equation}
where we have used $9(3d(1\!-\!p) \!+\! 1 ) N\alpha^2L^2\leq 27c^2p$ and $\chi - 135c^2 > \frac{1}{2}\chi$ in the last inequality.

Note that for any nonnegative sequence $(a_m)_{m\in\mathbb{N}}$ and $\rho\in(0,1)$, we have 
\begin{equation} \label{sum_sum}
	\begin{aligned}	\sum_{\tau=1}^{k}\sum_{m=0}^{\tau-1}\rho^ma_{\tau-m-1}  ={} & \sum_{\tau=1}^{k}\sum_{m=0}^{\tau-1} \rho^{\tau-m-1}a_m \\  
		={}& \sum_{m=0}^{k-1} \rho^{-m-1}a_m \sum_{\tau=m+1}^{k} \rho^{\tau}
        \leq \frac{1}{1-\rho}\sum_{m=0}^{k-1}a_m .
	\end{aligned}
\end{equation}
Consequently, we have
\begin{equation}  \label{sumEfFinal}
    \begin{aligned}
        \sum_{\tau=0}^{k-1} E_f^{\tau} \leq{} & 
        \frac{2}{\chi p}E_f^0 + \frac{108Nc^2}{\chi } \sum_{m=0}^{k-1} \mathbb{E} \!\left[ \big\| \nabla f(\bar{x}^m) \big\|^2 \right] + \frac{9L^2C_u}{ 4\chi p}\sum_{m=0}^{k-1} \sum_i(u_i^m)^2 \\
        \leq{} & \frac{3}{\chi p} E_f^0 + \frac{216Nc^2}{\chi\alpha} \delta^0 + \frac{5NL^2}{2 \chi} R_u,
    \end{aligned}
\end{equation}
where we have used~\eqref{sum_sum} in the first inequality, and~\eqref{nablafMid} with $c\leq\frac{(1-\sigma^2)^2}{28^2}$ in the last inequality. Then, using the inequality~\eqref{ExEcComp} and $E_f^0\leq NR_0$, we derive from~\eqref{sumEfFinal} that
\begin{equation*}
        \frac{1}{N}\sum_{\tau=0}^{k-1} E_x^{\tau} \leq \frac{216c^2}{\alpha L^2\chi } \delta^0 + \frac{3R_0}{\chi pL^2} +  \frac{5R_u}{ 2\chi},
\end{equation*}
which completes the proof of~\eqref{eq:smooth_x_consensus_rate}.

From the definition of $E_f^k$ in Lemma~\ref{Sigma_inequality} and the condition on $\alpha$, we can also derive that $4L^2E_x^k+E_s^k \leq \frac{1}{4\alpha L} E_f^k$. Consequently, we have
\begin{equation*} \label{final_Es}
	\begin{aligned}
		& \frac{1}{N}\sum_{\tau=0}^{k-1} \mathbb{E}\!\left[ \left\| s^{\tau} - \mathbf{1}_N\otimes \nabla f(\bar{x}^{\tau}) \right\|^2 \right] \\
		\leq{} & \frac{3}{2N}\sum_{\tau=0}^{k-1} \mathbb{E}\!\left[ \left\| s^{\tau} \!-\! \mathbf{1}_N\otimes \bar{g}^{\tau} \right\|^2 \right] + 3\sum_{\tau=0}^{k-1} \mathbb{E} \!\left[ \left\| \bar{g}^{\tau} \!-\! \nabla f(\bar{x}^{\tau}) \right\|^2 \right] \\
		\leq {} & \frac{3}{2N}\!\sum_{\tau=0}^{k-1}\!E_s^{\tau} \!+\! 3\sum_{\tau=0}^{k-1}\!\left(  \frac{2E_g^k}{N} \!+\! \frac{2L^2E_x^k}{N} \right)\!
        \leq \frac{3\left(\frac{1}{4\alpha L}\!+\!4\right)}{2N}\!\sum_{\tau=0}^{k-1}
        \!E_f^k \\
        \leq{} & \frac{108c^2}{\chi\alpha^2L} \delta^0 + \frac{3}{2\alpha L\chi p} R_0 + \frac{5L}{4\alpha \chi} R_u,
	\end{aligned}
\end{equation*}	
where we have used~\eqref{fbarTemp} in the second inequality, and $1/(4\alpha L)+4<1/(3\alpha L)$ in the last inequality. The proof for the consensus errors of Theorem~\ref{theorem1} can now be concluded.

\subsection{Proof of Theorem 2} \label{proof_theorem2}
We derive from~\eqref{delta_k} that
\begin{equation} \label{deltakkPL}
    \begin{aligned}
        \delta^{k+1} \leq{} &  \delta^k \!-\! \frac{\alpha}{2}\mathbb{E}\!\left[\big\| \nabla f(\bar{x}^k) \big\|^2\right] \!-\! \left(\ext{-2}
        \frac{1}{2\alpha} \!-\! \frac{L}{2}
        \ext{-2}\right)\!\mathbb{E}\!\left[\left\| \bar{x}^{k+1} \!-\! \bar{x}^k \right\|^2\right] \\
        & + \frac{\alpha}{N} \left(E_f^k + L^2\cdot \frac{1}{L^2} E_f^k \right) \\
        \leq{} & ( 1-\alpha\mu )\delta^k + \frac{2\alpha}{N} E_f^k \!-\! \left(\frac{1}{2\alpha} \!-\! \frac{L}{2}\right)\!\mathbb{E}\!\left[\left\| \bar{x}^{k+1} \!-\! \bar{x}^k \right\|^2\right]\!,
    \end{aligned}
\end{equation}
where we have used~\eqref{ExEcComp} and $E_g^k\leq E_f^k$ in the first inequality, and the PL condition~\eqref{mu_gradient} in the last inequality.

Combining~\eqref{deltakkPL} and~\eqref{EfLemmaForm}, we can get
\begin{equation} \label{LMIEfDelta}
    \begin{aligned}
        \begin{bmatrix}
           \frac{4\alpha}{\chi pN} E_f^{k+1} \\ \delta^{k+1}
        \end{bmatrix}
        \leq{}{} &
        \begin{bmatrix}
            1 - \chi p & 0 \\
            \frac{1}{2}\chi p  &  1-\alpha\mu
        \end{bmatrix}
        \begin{bmatrix}
           \frac{4\alpha}{\chi pN} E_f^k \\ \delta^k
        \end{bmatrix} \\
        &+ 
        \begin{bmatrix}
           \frac{ 36( 3d(1-p) + 1 ) \alpha L^2}{\chi p} \mathbb{E}[\| \bar{x}^{k+1} - \bar{x}^k \|^2] +  \frac{9\alpha L^2C_u}{2\chi p N} \sum_i(u_i^k)^2 \\
            - (\frac{1}{2\alpha} - \frac{L}{2})\mathbb{E}[\| \bar{x}^{k+1} - \bar{x}^k \|^2]
        \end{bmatrix},
    \end{aligned}
\end{equation}
in which the inequality is to be interpreted component-wise. By denoting  $E_d^k = \delta^k +  \frac{4\alpha}{\chi pN} E_f^k$, we can derive from~\eqref{LMIEfDelta} that
\begin{equation} \label{Edkk}
    \begin{aligned}
        E_d^{k+1} \!\leq{} & (1\!-\!\alpha\mu)\delta^k \!+\! \left(\!1\!-\!\frac{1}{2}\chi p\!\right)\! \frac{4\alpha E_f^k}{\chi pN}  + \frac{9\alpha L^2C_u\!\sum\nolimits_i(u_i^k)^2}{2\chi pN}  \\
        & \!\!- \!\left[\ext{-2} \frac{1}{2\alpha} \!-\! \frac{L}{2} \!-\!  \frac{36( 3d(1\!-\!p) \!+\! 1 )\alpha L^2}{\chi p} \right]\!\mathbb{E}\!\left[\ext{-2}\big\| \bar{x}^{k+1} \!-\! \bar{x}^k \big\|^2\right] \\
        \leq{} & \lambda E_d^k + \frac{9\alpha L^2C_u}{2\chi pN}\sum\nolimits_i(u_i^k)^2,
    \end{aligned}
\end{equation}
where in the second step, we use $\frac{1}{2\alpha} - \frac{L}{2} -  \frac{36(3d(1\!-\!p) \!+\! 1 )\alpha L^2}{\chi p} > 0$ that follows from $\alpha L = c\sqrt{\frac{p}{d(1-p)+1}}$ and $c\leq\frac{(1-\sigma^2)^2}{28^2}$. We further derive from~\eqref{Edkk} by induction that
\begin{equation} \label{Edkkkk}
    \begin{aligned}
        E_d^k \leq \lambda^kE_d^0 + \frac{9\alpha L^2C_u}{2\chi pN} \sum_{m=0}^{k-1}\lambda^m  \sum_i(u_i^{k-m-1})^2.
    \end{aligned}
\end{equation}
Now, using~\eqref{ExEcComp} and the definition of $E_d$, we have $E_x^k \leq \frac{\chi pN}{4\alpha L^2}E_d^k$ and $\delta^k \leq E_d^k$. The bound for $\delta^k$ and $\frac{1}{N} \sum_{i=1}^N  \mathbb{E} [ \|x_i^k - \bar{x}^k \|^2 ] $ in Theorem~\ref{theorem2} then follow from~\eqref{Edkkkk}.

Similarly, we derive that 	
\begin{equation*}
	\begin{aligned}
		& \frac{1}{N}  \mathbb{E} \!\left[ \left\| s^{k} - \mathbf{1}_N\otimes \nabla f(\bar{x}^k) \right\|^2 \right] \\
		\leq{} &  \frac{3}{2N} \mathbb{E} \!\left[ \left\| s^{k} - \mathbf{1}_N\otimes \bar{g}^k \right\|^2 \right] + 3 \mathbb{E}\!\left[ \left\| \bar{g}^k - \nabla f(\bar{x}^k) \right\|^2 \right] \\
		\leq{} & \frac{3}{2N}E_s^{k} + 3\left( \frac{2}{N}E_g^k + \frac{2L^2}{N}E_x^k \right) \\
        \leq{} & \frac{3\left(\frac{1}{4\alpha L}\!+\!4\right)}{2N}E_f^k
        \leq{}
        \frac{1}{2\alpha LN}\frac{\chi pN}{4\alpha}E_d^k \\
        \leq{} & \frac{\chi p\lambda^k}{8\alpha^2 L}E_d^0 + \frac{9LC_u}{16\alpha N} \sum_{m=0}^{k-1}\lambda^m  \sum_i(u_i^{k-m-1})^2.
	\end{aligned}
\end{equation*}	
The proof is now complete.

\section{Numerical Simulations}

\subsection{Simulation on a Synthetic Test Case}
We consider a multi-agent nonconvex optimization problem adapted from \cite{9199106} with $N=50$ agents in the network, and the objective function of each agent is given as follows:
\begin{equation}
    \begin{aligned}
        f_i(x) = \frac{\alpha_i}{1+e^{-\xi_i^T x - v_i}} + \beta_i \ln (1+\|x\|^2),
    \end{aligned}
\end{equation}
where $\alpha_i, \beta_i, v_i \in \mathbb{R}$ are randomly generated parameters satisfying $\frac{1}{N}\sum_i \beta_i = 1 $, each $\xi_i \in \mathbb{R}^d$ is also randomly generated, and the dimension $d$ is set to $64$.

For the following numerical simulation of Algorithm 1, we set the step-size $\alpha = 0.02$ and the smoothing radius $u_i^k = 3/k^{\frac{3}{4}}$. All agents start from the same initial points to ensure consistency in the initial conditions across the network.

\subsubsection{Comparison with Other Algorithms}
Fig.~\ref{comparison_VRGE} compares Algorithm 1 with DGD-2p, GT-$2d$~\cite{9199106}, ZONE-M~\cite{hajinezhad2019zone} (with $J=100$), and DZO~\cite{yi2021linear}. In the figure, the probability used for Algorithm 1 is $p=0.1$. The horizontal axis is normalized and represents the sampling number $m$ (i.e., the number of zeroth-order queries). The two sub-figures illustrate the stationarity gap $\| \nabla f(\bar{x}^k)\|^2$ and the consensus error $\frac{1}{N}\sum_{i}\| x_i^k - \bar{x}^k \|^2$, respectively. 

By inspecting Fig.~\ref{comparison_VRGE}, we first see that the stationarity gap of DGD-2p converges faster than ZONE-M with $J=100$ and DZO, but they have generally similar convergence behavior. When comparing DGD-2p and GT-$2d$, we can see a clear difference between their convergence behavior: DGD-2p achieves fast convergence initially but slows down afterwards due to the inherent variance of the 2-point gradient estimator, whereas GT-$2d$ achieves higher eventual accuracy but slower initial convergence before approximately $1.5\times 10^4$ zeroth-order queries due to the higher sampling burden of the $2d$-point gradient estimator.

As demonstrated in Fig.~\ref{comparison_VRGE}, Algorithm 1 offers both high eventual accuracy and a fast convergence rate in terms of stationarity gap and consensus error. This improvement is attributed to the variance reduction mechanism employed in designing VR-GE, which effectively balances the sampling number and expected variance, thereby addressing the trade-off between convergence rate and sampling cost per zeroth-order gradient estimation that exists in current zeroth-order distributed optimization algorithms.

\begin{figure}[!t]
\centerline{\includegraphics[width=.55\columnwidth]{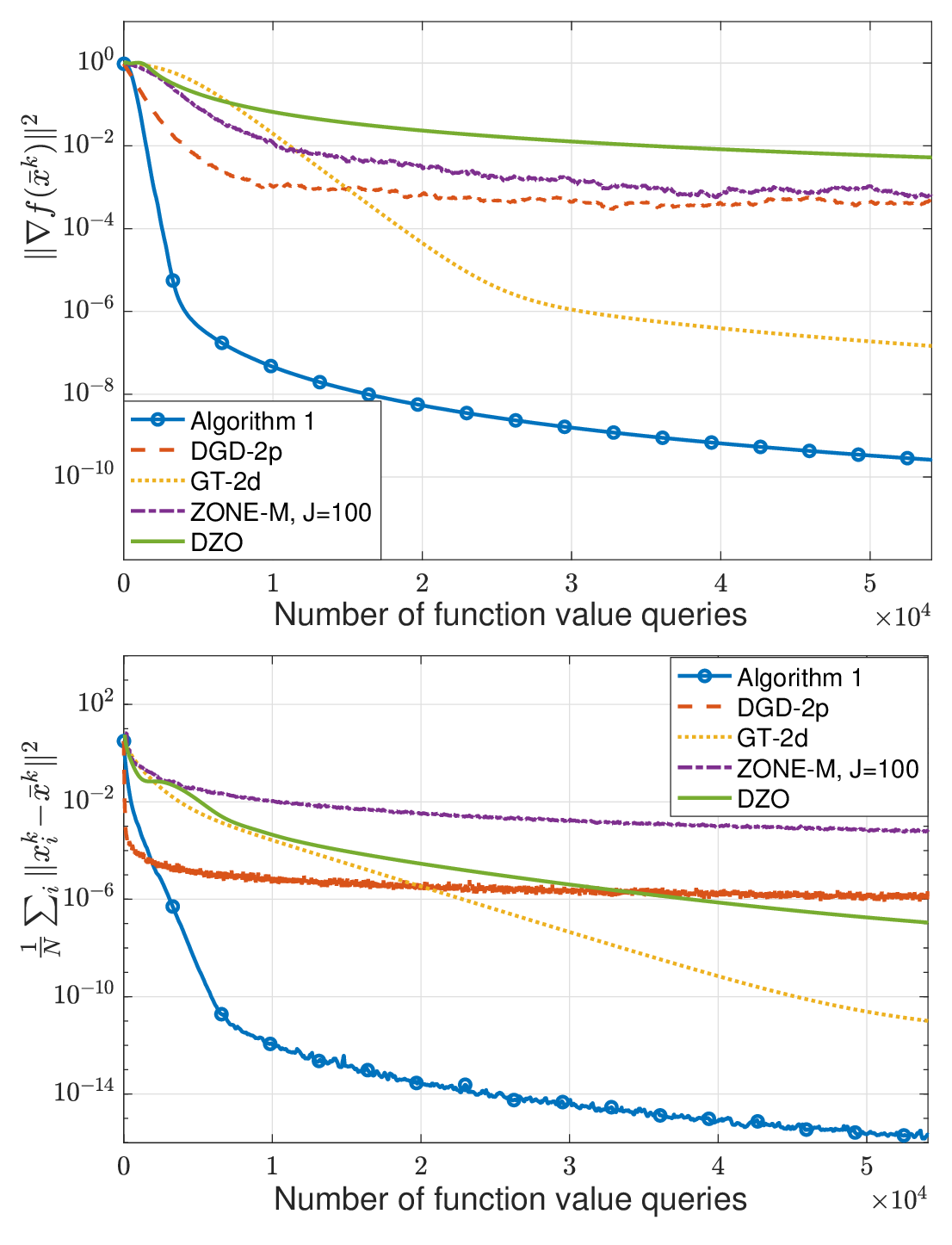}}
\caption{Convergence of Algorithm 1, ZONE-M with $J =100$, DGD-2p, GT-$2d$.}
\label{comparison_VRGE}
\end{figure}

\subsubsection{Comparison of Algorithm 1 with Different Probabilities}
Fig.~\ref{different_prob} compares the convergence of Algorithm 1 under different choices of the probability $p$, which reflects the frequency with which each agent takes snapshots. The three sub-figures illustrate the stationarity gap $\| \nabla f(\bar{x}^k)\|^2$, the consensus error $\frac{1}{N}\sum_{i}\| x_i^k - \bar{x}^k \|^2$, and the tracking error $\frac{1}{N}\sum_{i}\| s_i^k - \nabla f( \bar{x}^k )\|^2$, respectively. 

The results demonstrate that Algorithm 1 with a lower probability achieves better accuracy with fewer sampling numbers. However, a lower probability also results in more fluctuation during convergence. This is expected because, with a lower probability, the snapshot variables are updated less frequently, leading to a greater deviation from the true gradient as iterations progress.

Two notable cases are $p=0$ and $p=1$. When $p=1$, Algorithm 1 behaves the same as GT-$2d$, utilizing a $2d$-point gradient estimator at each step. This leads to inferior empirical convergence performance compared to when $p\in(0,1)$. On the other hand, with $p=0$, agents avoid using the $2d$-point estimation and opt to update only one random direction per iteration based on the initial gradient estimation. This approach leads to persistent variance and decreased convergence accuracy, as demonstrated in Fig.~\ref{different_prob}.

\begin{figure}[t]
\centerline{
\includegraphics[width=.65\columnwidth]{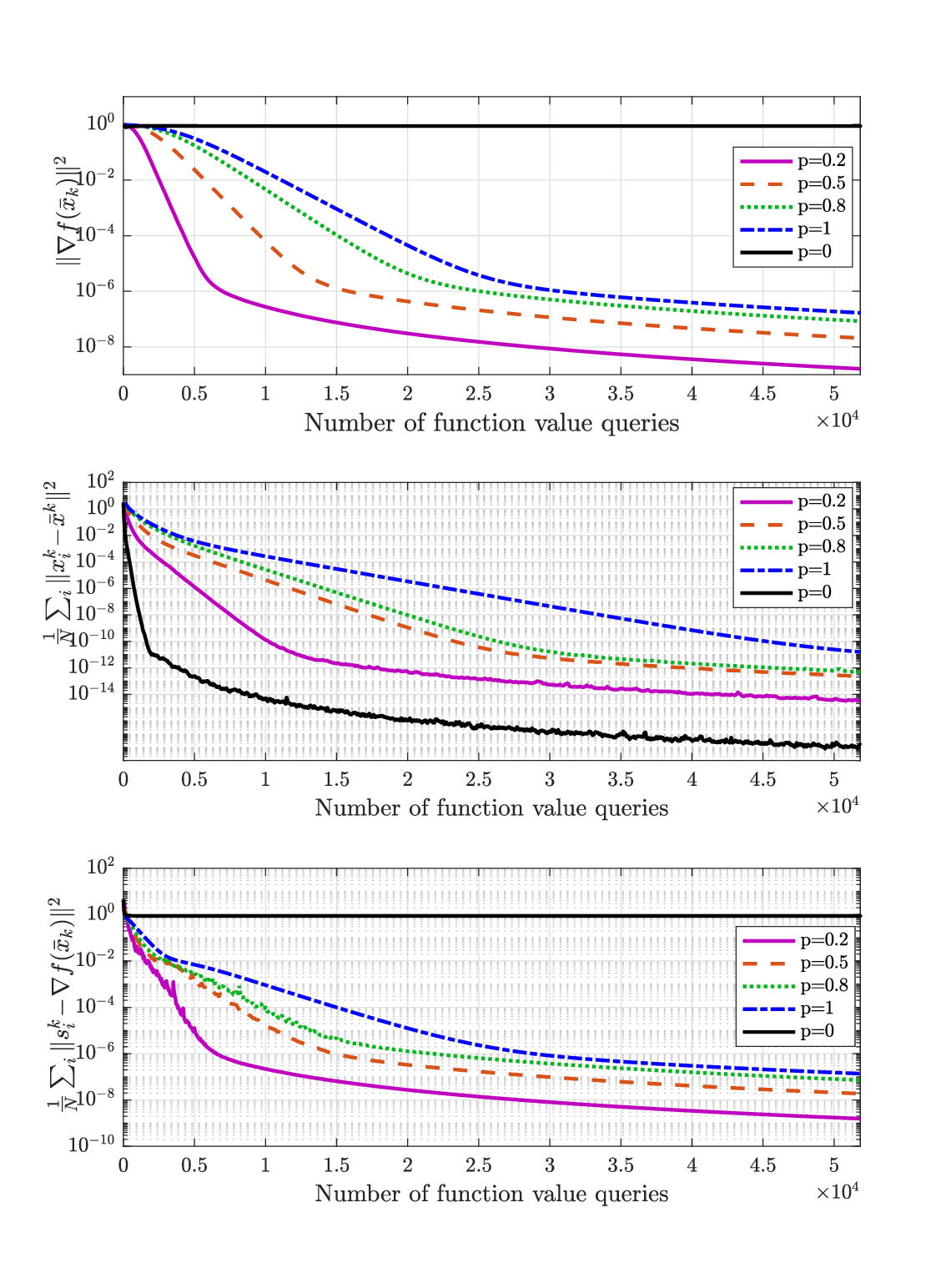}}
\vspace{-24pt}
\caption{Convergence of Algorithm 1 under probability $p$ = 0.2, 0.5, 0.8, and 1.}
\label{different_prob}
\end{figure}

\subsubsection{Comparison of Algorithm 1 under Different Dimensions}
Fig.~\ref{different_dim} compares the convergence of Algorithm 1 across different agent dimensions, alongside varying probabilities for taking snapshots within the algorithm. The results show that Algorithm 1 can effectively handle different scenarios, such as when $d=300$, achieving stationarity gaps that are below $10^{-6}$.

As the dimension increases, VR-GE requires more samples to accurately estimate the gradient. To maintain similar convergence performance across higher dimensions, the probability $p$ for taking snapshots can be adjusted to lower values. As shown in Fig.~\ref{different_dim}, decreasing the probability as the dimension grows allows Algorithm 1 to achieve a convergence rate and optimization accuracy that are comparable to cases with lower dimensions. However, this adjustment also leads to increased fluctuation during the convergence process. This fluctuation is a result of the randomness introduced by the snapshot mechanism.

\begin{figure}[!t]
\centerline{\includegraphics[width=.65\columnwidth]{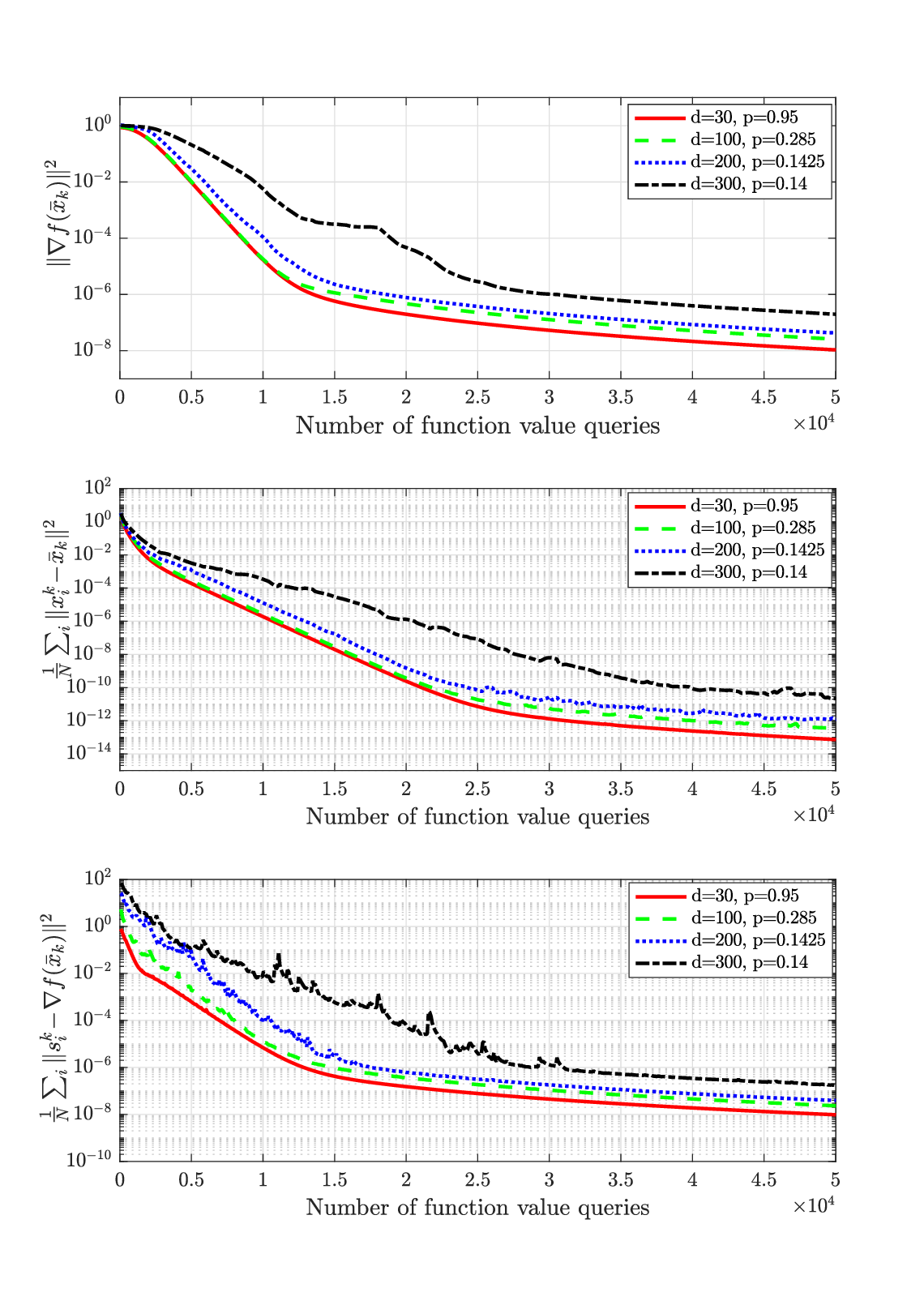}}
\vspace{-24pt}
\caption{Convergence of Algorithm 1 with different dimension $d$ = 30, 100, 200, and 300.}
\label{different_dim}
\end{figure}

\subsection{Simulation on a Test Case with Real World Data}
We consider an image classification problem employing the CIFAR-10 dataset~\cite{krizhevsky2009learning} to assess our algorithm's performance.
The setup involves $N=50$ parallel, independent agents interconnected via an undirected graph $\mathcal{G}$, with each agent handling a separate batch consisting of 200 samples. The associated weight matrix for graph $\mathcal{G}$ is generated by randomly sampling the nodes onto the unit sphere $\mathbb{S}^2$. An edge $(i,j)\in\mathcal{E}$ exists if the spherical distance between the corresponding agents is less than $\frac{3\pi}{4}$.  The doubly-stochastic mixing matrix $W$ is then constructed according to the Metropolis-Hastings rule~\cite{xiao2005scheme}. The objective function of each agent is a regularized version of the cross-entropy loss computed over its local dataset:
\begin{equation}
    \begin{aligned}
        F_i(\Theta) = \frac{1}{n_i}\sum_{k=1}^{n_i}l(\Theta;(x_k^{(i)}, y_k^{(i)})) + \frac{\lambda}{2} \ln (1 + \| \Theta \|_F^2),
    \end{aligned}
\end{equation}
where $\Theta\in\mathbb{R}^{q\times c} $ represents the global model parameter matrix to be optimized.
Here, the feature dimension is $q=65$ (64 CNN-extracted features plus 1 bias term) and the number of classes is $c=10$, yielding a total parameter dimension of $d = q\times c = 650$. Every node $i$ contains $n_i=200$ training samples, with $(x_k^{(i)}, y_k^{(i)})$ as the $k$-th feature vector and label at node $i$. The function $l(\cdot)$ is multi-class cross-entropy loss of the following  form:
\begin{equation}
l(\Theta; (x_k^{(i)}, y_k^{(i)})) = -\ln \left(\frac{\exp(\theta_{y_k^{(i)}}^T x_k^{(i)})}{\sum_{j=1}^{c} \exp(\theta_j^T x_k^{(i)})}\right).
\end{equation}
The regularization coefficient is set to $\lambda = 0.02$.

We set the probability parameter in Algorithm 1 to $p=2\times10^{-3}$. The stepsizes for Algorithm 1, DGD-2p, and GT-$2d$ are set to $\alpha = 3\times 10^{-4}$, $\eta = 1\times10^{-3}/\sqrt{k}$, and $\alpha = 5\times 10^{-3}$, respectively. For ZONE-M, we set $J=50$. For DZO, we set $\beta = 1\times 10^{-1}$, $\alpha = 1.5\times 10^{-1}$, and $\eta = 5\times 10^{-3}$.

We employ two metrics to evaluate the performance of algorithms: 
i) Squared gradient norm in Fig.~\ref{cifar}, defined as $ \| \frac{1}{N}\sum_{i=1}^{N}\nabla F_i(\bar{\Theta}) \|^2$, corresponds to the squared gradient norm of the global objective function evaluated on the entire training set that demonstrate the optimization convergence.
This metric reflects the convergence behavior of the optimization process, where $\bar{\Theta} = \frac{1}{N}\sum_{j=1}^{N} \Theta_j$ denotes the global average of the model parameters across all nodes. ii) Consensus error in Fig.~\ref{cifar_consensus}, defined as $\sum_{i=1}^N\| \Theta_i - \bar{\Theta} \|^2$, measures the total deviation of all node parameters from their global average and tracks the algorithm's progress toward consensus.

\begin{figure}[htb]
\centerline{\includegraphics[width=.6\columnwidth]{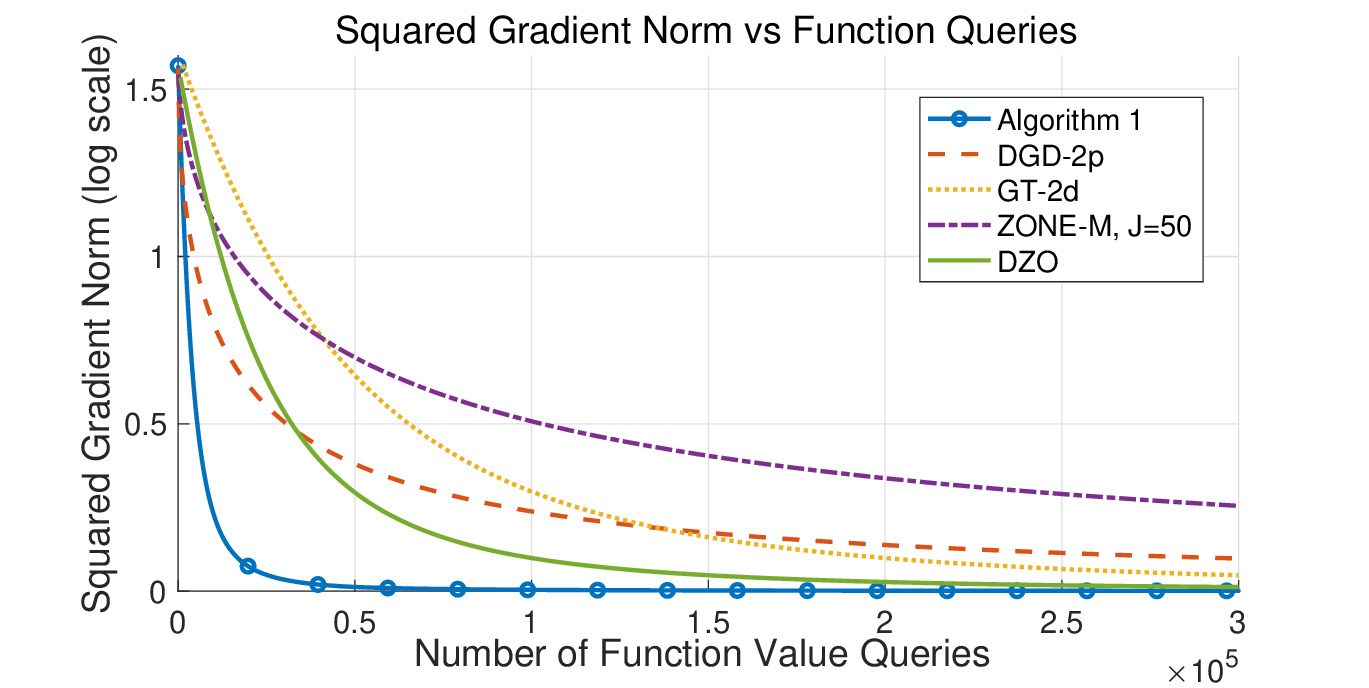}}
\caption{Squared gradient norm curves of Algorithm 1, ZONE-M, DGD-2p, GT-$2d$, and DZO on the CIFAR-10 dataset.}
\label{cifar}
\end{figure}

\begin{figure}[htb]
\centerline{\includegraphics[width=.6\columnwidth]{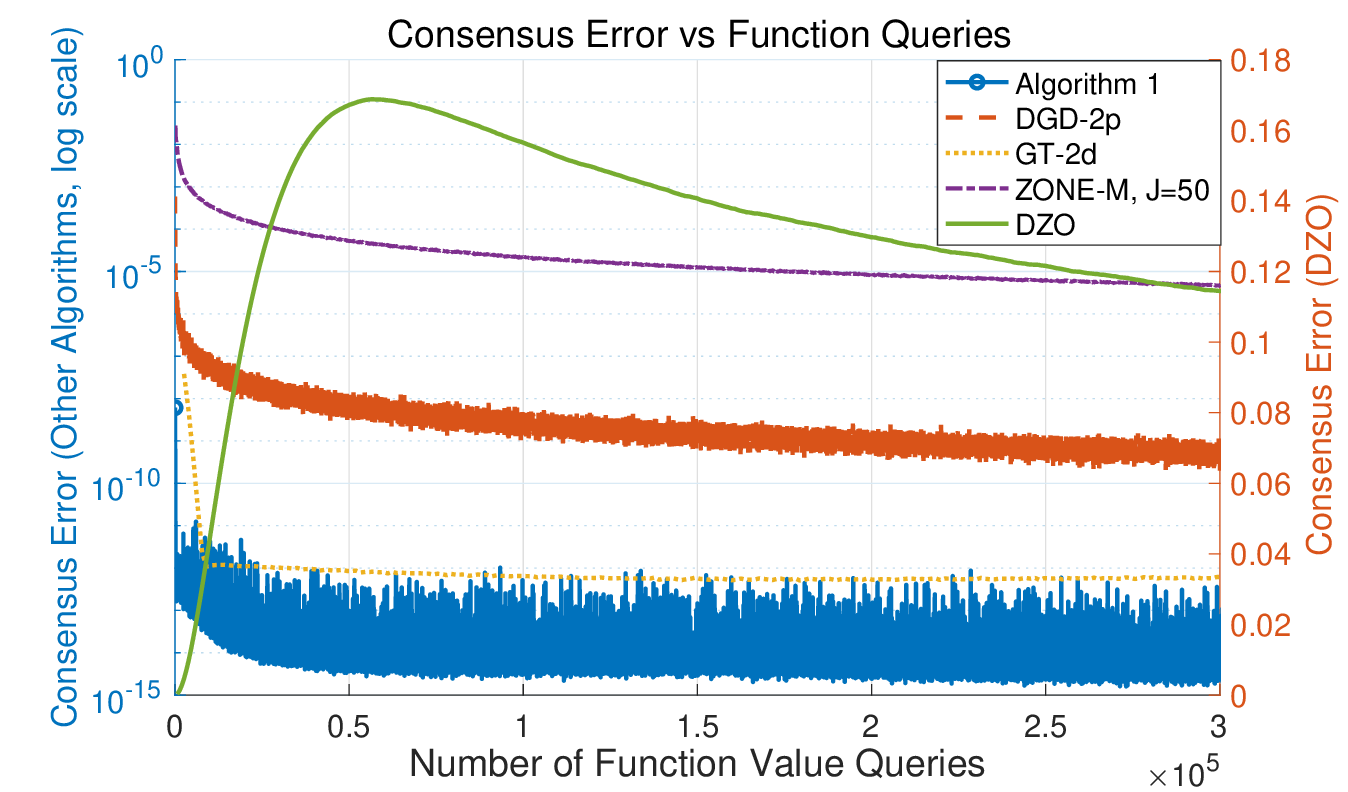}}
\caption{Consensus errors of Algorithm 1, ZONE-M, DGD-2p, GT-$2d$, and DZO on the CIFAR-10 dataset.}
\label{cifar_consensus}
\end{figure}

As shown in Fig.~\ref{cifar}, Algorithm 1 achieves a faster convergence rate than all other algorithms, and higher convergence accuracy than both DGD-2p and ZONE-M.

In Fig.~\ref{cifar_consensus}, we use dual-axis plot to show the consensus error. The right y-axis displays the error for the DZO algorithm on a linear scale, while the left y-axis, in logarithmic scale, corresponds to the consensus error for the other algorithms. While the DZO algorithm demonstrates a relatively fast convergence rate in Fig.~\ref{cifar}, its steady-state consensus error remains higher, around $10^{-1}$, compared to the other algorithms. In contrast, Algorithm 1 and GT-$2d$ achieve a much lower consensus error, converging to approximately $10^{-13}$. The DGD-2p and GT-$2d$ algorithms approach $10^{-9}$ and $10^{-5}$, respectively. The trajectories of DGD-2p and Algorithm 1 has more fluctuations. This behavior is due to the stochasticity present in their state update processes.

The code for the numerical simulations can be found at \url{https://github.com/HuaiyiMu/VR-GE}.

\section{Conclusion}
In this paper, we proposed an improved variance-reduced gradient estimator and integrated it with gradient tracking mechanism for nonconvex distributed zeroth-order optimization problems. Through rigorous analysis, we demonstrated that our algorithm achieves sublinear convergence for smooth nonconvex functions that is comparable with first-order gradient tracking algorithms, while maintaining relatively low sampling cost per gradient estimation. We also derive linear convergence rate for smooth nonconvex and gradient dominated objective functions. Comparative evaluations with existing distributed zeroth-order optimization algorithms verified the effectiveness of the proposed gradient estimator.

\appendix

\section{Auxiliary Lemmas}

This section summarizes some auxiliary lemmas for convergence analysis.

\begin{lemma}[\cite{nesterov2018lectures}] \label{LSmoothSquare}
    Suppose $f : \mathbb{R}^d \to \mathbb{R}$ is $L$-smooth. Then for any $x,y \in \mathbb{R}^d$, we have
    \begin{equation}
        \begin{aligned}
            f(y) \leq f(x) + \langle  \nabla f(x) , y-x \rangle + \frac{L}{2}\|y-x\|^2.
        \end{aligned}
    \end{equation}
    
\end{lemma}

\begin{lemma}[\cite{9199106}] \label{2d_error}
	Let $f : \mathbb{R}^d \to \mathbb{R}$ be $L$-smooth. Then for any $x\in \mathbb{R}^d$, 
	\begin{equation}
		\left\| G_f^{(2d)}(x,u) - \nabla f(x) \right\| \leq \frac{1}{2}uL\sqrt{d}.
	\end{equation}
\end{lemma}

\begin{lemma}[\cite{qu2017harnessing}] \label{sigma_leq1}
	Let $\sigma \triangleq \| W - \frac{1}{N}\mathbf{1}_N\mathbf{1}_N^T \|_2<1$. For any $z_1, \cdots, z_N \in \mathbb{R}^d$, we have
	\[
	  \| (W \otimes I_d)(z - \mathbf{1}_N \otimes \bar{z}) \| \leq \sigma \| z - \mathbf{1}_N \otimes \bar{z} \|,
	\]
	where we denote $ z = \begin{bmatrix}
		z_1^T & \cdots & z_N^T
	\end{bmatrix}^T, \bar{z} = \frac{1}{N}\sum_{i=1}^N z_i $.
\end{lemma}

In the following, we shall denote the $\sigma$-algebra generated by $(x^\tau,s^\tau,g^\tau)_{\tau=0}^k$ by $\mathcal{F}^k$. Note that $x^{k+1}$ is $\mathcal{F}^k$-measurable.

\section{Proof of Lemma \ref{g^*-nabla_f}}
\label{appendix:proof_lemma_g^*-nabla_f}

Following from $\zeta_i^{k+1} \sim \mathrm{Ber}(p) $, we derive that
\begin{equation} \label{gikk}
    \begin{aligned}
        & \mathbb{E}\!\left[ \left\| g_i^{k+1} - \nabla f_i(x_i^{k+1}) \right\|^2 \right] \\ 
        ={} & \mathbb{E}\!\left[\mathbb{E}\Econd*{\left\| g_i^{k+1} - \nabla f_i(x_i^{k+1}) \right\|^2}{\mathcal{F}^{k},l_i^{k+1}} \right] \\
        ={} &  (1\!-\!p)\mathbb{E}\Big[ \big\| g_i^k \!+\! G_{f_i}^{(c)}(x_i^{k+1}, u_i^{k+1}, l_i^{k+1} ) \!-\!G_{f_i}^{(c)}(x_i^{k}, u_i^{k}, l_i^{k+1} ) \\
        & \!-\!\nabla \! f_i(x_i^{k+1}) \big\|^2 \Big]
        \!+\! p \ext{1} \mathbb{E}\!\left[\big\| G_{f_i}^{(2d)}\!( x_i^{k+1}, u_i^{k+1} ) \!-\! \nabla\! f_i(x_i^{k+1}) \big\|^2\right] \\
        ={} & (1\!-\!p)\mathbb{E}\!\left[ \left\| g_i^k - \nabla f_i(x_i^{k}) \right\|^2 \right] + 2(1\!-\!p) \mathbb{E}\Big[\Big\langle g_i^k \!-\! \nabla f_i(x_i^{k}) , \\
        & \ \ \mathbb{E}\big[ G_{f_i}^{(c)}(x_i^{k+1}, u_i^{k+1}, l_i^{k+1} ) -G_{f_i}^{(c)}(x_i^{k}, u_i^{k}, l_i^{k+1})  \\
        & \qquad- \big( \nabla f_i(x_i^{k+1}) - \nabla f_i(x_i^{k}) \big)\big|\mathcal{F}^{k} \big] \Big\rangle\Big] \\
        &+ (1\!-\!p) \mathbb{E}\Big[ \big\|G_{f_i}^{(c)}(x_i^{k+1}, u_i^{k+1}, l_i^{k+1} ) -G_{f_i}^{(c)}(x_i^{k}, u_i^{k}, l_i^{k+1})  \\
        & \qquad\qquad-\! \big( \nabla f_i(x_i^{k+1}) \!-\! \nabla f_i(x_i^{k}) \big) \big\|^2 \Big] 
        +\frac{1}{4}pL^2d(u_i^{k+1})^2,
    \end{aligned}
\end{equation}
where we have used $\big\|G_h^{(2d)}(x,u) -\nabla h(x) \big\| \leq \frac{1}{2}u L\sqrt{d}$ for $h$ being $L$-smooth in the second equality.

We start from the second term on the RHS of~\eqref{gikk} and get
\begin{equation} \label{cross_term}
    \begin{aligned}
        & \Big\langle g_i^k - \nabla f_i(x_i^{k}) , \mathbb{E}\big[G_{f_i}^{(c)}(x_i^{k+1}, u_i^{k+1}, l_i^{k+1} ) \\
        & -G_{f_i}^{(c)}(x_i^{k},\! u_i^{k},\! l_i^{k+1})  - \big( \nabla f_i(x_i^{k+1}) \!-\! \nabla f_i(x_i^{k}) \big) \big|\mathcal{F}^{k}\big] \Big\rangle \\
        \leq{} & \big\| g_i^k - \nabla f_i(x_i^{k}) \big\| \!\cdot\! \Big\| \mathbb{E}\big[ G_{f_i}^{(c)}(x_i^{k+1}, u_i^{k+1}, l_i^{k+1} ) \\
        & -G_{f_i}^{(c)}(x_i^{k}, u_i^{k}, l_i^{k+1}) - \big( \nabla f_i(x_i^{k+1}) - \nabla f_i(x_i^{k}) \big)\big|\mathcal{F}^k \big] \Big\| \\
        ={} & \big\| g_i^k - \nabla f_i(x_i^{k}) \big\| \!\cdot\! \big\| \big( G_{f_i}^{(2d)}(x_i^{k+1}, u_i^{k+1}) - \nabla f_i(x_i^{k+1}) \big) \\
        & - \big( G_{f_i}^{(2d)}(x_i^k, u_i^k) - \nabla f_i(x_i^{k}) \big) \big\| \\
        \leq{} & \big\| g_i^k - \nabla f_i(x_i^{k}) \big\| \!\cdot\! \Big( \big\| G_{f_i}^{(2d)}(x_i^{k+1}, u_i^{k+1}) - \nabla f_i(x_i^{k+1}) \big\| \\
        & + \big\| G_{f_i}^{(2d)}(x_i^k, u_i^k) - \nabla f_i(x_i^{k}) \big\| \Big) \\
        \leq{} &  \big\| g_i^k - \nabla f_i(x_i^{k}) \big\| \!\cdot\! \left( \frac{1}{2}u_i^{k+1}L\sqrt{d} + \frac{1}{2}u_i^{k}L\sqrt{d} \right) \\
        \leq{} & \big\| g_i^k - \nabla f_i(x_i^{k}) \big\| \cdot u_i^{k}L\sqrt{d} \\
        \leq{} & \frac{p}{4} \big\| g_i^k - \nabla f_i(x_i^{k}) \big\|^2 + \frac{L^2d}{p}(u_i^{k})^2.
    \end{aligned}
\end{equation}
Here the Cauchy--Schwarz inequality $ | \langle u,v \rangle | \leq \|u\|\|v\| $ is applied in the first step; $ \mathbb{E}_{l\sim\mathcal{U}[d]}\big[G_h^{(c)}(x,u,l)\big] = G_h^{(2d)}(x,u) $ is used in the second step; the triangle inequality is employed in the third step; the fifth step follows from the condition that the sequence $(u_i^k)_k$ is non-increasing. Finally, the AM--GM inequality $2\sqrt{ab} \leq a + b$ is used in the last step.

For the third term on the RHS of~\eqref{gikk}, we note that 
\begin{equation} \label{gikkp}
    \begin{aligned}
        & \mathbb{E}\Big[ \big\| G_{f_i}^{(c)}(x_i^{k+1}, u_i^{k+1}, l_i^{k+1} ) -G_{f_i}^{(c)}(x_i^{k}, u_i^{k}, l_i^{k+1})  \\
        & \quad - \big( \nabla f_i(x_i^{k+1}) - \nabla f_i(x_i^{k}) \big) \big\|^2 \Big|\mathcal{F}^k\Big] \\ 
        \leq{} & 2\ext{2}\mathbb{E}\Big[\big\|G_{f_i}^{(c)}(x_i^{k+1}, u_i^{k+1}, l_i^{k+1} ) -G_{f_i}^{(c)}(x_i^{k}, u_i^{k}, l_i^{k+1}) \big\|^2\Big|\mathcal{F}^k\Big] \\
        & + 2\ext{2}\mathbb{E}\Big[\big\| \nabla f_i(x_i^{k+1}) - \nabla f_i(x_i^{k}) \big\|^2\Big|\mathcal{F}^k\Big] \\
        ={} & 2d\big\| G_{f_i}^{(2d)}(x_i^{k+1}, u_i^{k+1}) - G_{f_i}^{(2d)}(x_i^k, u_i^k) \big\|^2 \\
        & + 2\big\| \nabla f_i(x_i^{k+1}) - \nabla f_i(x_i^{k}) \big\|^2 \\ 
        \leq{} & 2d \big\| G_{f_i}^{(2d)}(x_i^{k+1}, u_i^{k+1}) - G_{f_i}^{(2d)}(x_i^k, u_i^k) \big\|^2 \\
        & + 2L^2 \| x_i^{k+1} \!- x_i^k \|^2,
    \end{aligned}
\end{equation}
where we have used $\|u-v\|^2 \leq 2\|u\|^2 + 2\|v\|^2 $ in the first step, and $L$-smoothness of $f_i$ in the last step.

For the first term on the RHS of~\eqref{gikkp}, we have 
\begin{equation} \label{2dkkk}
    \begin{aligned}
        & \big\| G_{f_i}^{(2d)}(x_i^{k+1}, u_i^{k+1}) - G_{f_i}^{(2d)}(x_i^k, u_i^k) \big\|^2 \\
        ={} & \big\| \big( G_{f_i}^{(2d)}(x_i^{k+1}, u_i^{k+1}) - \nabla f_i(x_i^{k+1}) \big) - \big( G_{f_i}^{(2d)}(x_i^k, u_i^k) \\
        & \ \ - \nabla f_i(x_i^k) \big) + \big( \nabla f_i(x_i^{k+1}) - \nabla f_i(x_i^k) \big) \big\|^2 \\
        \leq{} & (u_i^{k+1})^2L^2d + (u_i^k)^2L^2d + 2\| \nabla f_i(x_i^{k+1}) - \nabla f_i(x_i^k) \|^2 \\
        \leq{} & 2L^2d(u_i^k)^2 + 2L^2 \| x_i^{k+1} - x_i^k \|^2,
    \end{aligned}
\end{equation}
where we have used $\|a+b+c\|^2 \leq 4\|a\|^2 + 4\|b\|^2 + 2\|c\|^2 $ in the second step, and the last step follows from the monotonicity of the sequence $(u_i^k)_k$ and $L$-smoothness of $f_i$.

Combining the inequalities~\eqref{2dkkk}, \eqref{gikkp} and \eqref{cross_term}, taking the total expectation, and plugging the outcomes into~\eqref{gikk}, we get
\[
\begin{aligned}
& \mathbb{E}\!\left[ \left\| g_i^{k+1} - \nabla f_i(x_i^{k+1}) \right\|^2 \right] \\
\leq{} & (1-p)\!\left(1+\frac{p}{2}\right)\mathbb{E}\!\left[ \left\| g_i^{k} - \nabla f_i(x_i^{k}) \right\|^2 \right] \\
& + (4d+2)(1-p)L^2\,\mathbb{E}\!\left[\left\| x_i^{k+1} - x_i^k \right\|^2\right] \\
&+\left(4d^2(1-p)+\frac{2(1-p)d}{p}
+\frac{pd}{4}\right)(Lu_i^k)^2 \\
\leq{} &
(1-p)\!\left(1+\frac{p}{2}\right)\mathbb{E}\!\left[ \left\| g_i^{k} - \nabla f_i(x_i^{k}) \right\|^2 \right] \\
& + 6d(1-p)L^2\,\mathbb{E}\!\left[\left\| x_i^{k+1} - x_i^k \right\|^2\right] +C_u(Lu_i^k)^2
\end{aligned},
\]
which completes the proof.

\section{Proof of Lemma \ref{error_optimization}}
\label{appendix:proof_lemma_error_optimization}

First, by left multiplying $\frac{1}{N}\mathbf{1}_N^T \otimes I_d $ on both sides of~\eqref{algorithm_compact_2}, and using the double stochasticity of $W$ and the initialization $s^0 = g^0$, we obtain
	\begin{equation*}
	\bar{s}^k = \bar{g}^k,		
	\end{equation*}
where $\bar{s}^k = \frac{1}{N}\sum_{i=1}^N s_i^k $ and $ \bar{g}^k = \frac{1}{N}\sum_{i=1}^N g_i^k $.

Then, from~\eqref{algorithm_compact_1}, we get
\begin{equation} \label{xbar_kplus1}
	\bar{x}^{k+1}=\bar{x}^k - \alpha \bar{g}^k. 
\end{equation}
Leveraging the $L$-smoothness of the function $f$, we have 
\begin{align} 
        f(\bar{x}^{k+1})-f(\bar{x}^k)
        \leq{} & \langle \nabla f(\bar{x}^k), \bar{x}^{k+1} - \bar{x}^k  \rangle + \frac{L}{2}\| \bar{x}^{k+1} - \bar{x}^k \|^2 \nonumber\\
        ={} & - \alpha \langle \nabla f(\bar{x}^k) - \bar{g}^k,  \bar{g}^k  \rangle - \left( \frac{1}{\alpha} - \frac{L}{2} \right)\| \bar{x}^{k+1} - \bar{x}^k \|^2 \label{fBarOri}
\end{align}
where we have used Lemma~\ref{LSmoothSquare} in the first step, and~\eqref{xbar_kplus1} in the second step, 

For the first term in~\eqref{fBarOri}, it is not hard to verify that 
\begin{equation} \label{fbarMiddle}
    \begin{aligned}
        2\langle \nabla f(\bar{x}^k) \!-\! \bar{g}^k,  \bar{g}^k  \rangle
        ={} & \|  \nabla f(\bar{x}^k) \|^2 - \| \nabla f(\bar{x}^k) \!-\! \bar{g}^k \|^2 \\
        & - \frac{1}{\alpha^2}\| \bar{x}^{k+1} \!-\! \bar{x}^k \|^2.
    \end{aligned}
\end{equation}
Plugging~\eqref{fbarMiddle} into the inequality~\eqref{fBarOri} and taking expectations on both sides, we get
\begin{equation} \label{fbarMiddd}
    \begin{aligned}
        \mathbb{E}\big[f(\bar{x}^{k+1})-f(\bar{x}^k)\big]
        \leq{} & - \frac{\alpha}{2} \mathbb{E} \big[\|  \nabla f(\bar{x}^k) \|^2\big] \!-\! \left( \frac{1}{2\alpha} \!-\! \frac{L}{2} \right)\!\mathbb{E}\big[\| \bar{x}^{k+1} \!-\! \bar{x}^k \|^2\big] \\
        &  + \frac{\alpha}{2}\mathbb{E}\big[\| \nabla f(\bar{x}^k) - \bar{g}^k \|^2\big]
    \end{aligned}
\end{equation}
For the last term on the RHS of~\eqref{fbarMiddd}, have
\begin{equation*}
    \begin{aligned}
        \mathbb{E}\big[\| \nabla f(\bar{x}^k) - \bar{g}^k \|^2\big]
        ={} & \mathbb{E}\!\left[
        \Big\|\mfrac{1}{N}\sum\nolimits_{i=1}^N \big(\nabla f_i(\bar{x}^k)-g_i^k\big)\Big\|^2\right]\\
        \leq{} &
        \frac{1}{N}\sum\nolimits_{i=1}^N
        \mathbb{E}\!\left[
        \left\|\left( \nabla\ext{-2} f_i (x_i^k) \!-\! g_i^k \right)+\left( \nabla\ext{-2} f_i (\bar{x}^k) \!-\! \nabla\ext{-2} f_i (x_i^k) \right)\right\|^2\right] \\
        \leq{} &  \frac{2}{N}\sum\nolimits_{i=1}^N\mathbb{E}\!\left[ \left\| \nabla f_i (x_i^k) - g_i^k \right\|^2+ \left\| \nabla f_i (\bar{x}^k) -\nabla f_i (x_i^k) \right\|^2 \right] \\
        \leq{} & \frac{2}{N}E_g^k + \frac{2L^2}{N}E_x^k,
    \end{aligned}
\end{equation*}
which also proves~\eqref{fbarTemp}. Combining~it with~\eqref{fbarMiddd}, we complete the proof.

\section{Proof of Lemma~\ref{LMI_lemma}}
\label{appendix:proof_lemma_LMI_lemma}

First, following from \eqref{algorithm_compact} and \eqref{xbar_kplus1}, we derive that 
\begin{equation} \label{E_x^k+1}
	\begin{aligned}
		E_x^{k+1}  ={} &  \mathbb{E}\!\left[ \left\| x^{k+1} - \mathbf{1}_N \otimes \bar{x}^{k+1} \right\|^2 \right] \\
		={} &  \mathbb{E}\!\left[ \left\| (W \otimes I_d)\big[x^k \!-\! \mathbf{1}_N \otimes \bar{x}^k \!-\! \alpha (s^k \!-\! \mathbf{1}_N \otimes \bar{g}^k) \big] \right\|^2 \right] \\
        \leq {} &\sigma^2\,
        \mathbb{E}\!\left[ \left\| x^k \!-\! \mathbf{1}_N \otimes \bar{x}^k \!-\! \alpha (s^k \!-\! \mathbf{1}_N \otimes \bar{g}^k) \right\|^2 \right] \\
		\leq{} & \sigma^2\left(1+\frac{1-\sigma^2}{3\sigma^2}\right) \mathbb{E} \!\left[ \left\| x^k - \mathbf{1}_N\otimes \bar{x}^k \right\|^2 \right] \\
		& + \sigma^2\left(1+\frac{3\sigma^2}{1-\sigma^2}\right)\alpha^2 \mathbb{E} \!\left[ \left\| s^k - \mathbf{1}_N\otimes \bar{g}^k \right\|^2 \right] \\
		\leq{} & \frac{1+2\sigma^2}{3}E_x^k + \frac{3}{1-\sigma^2}\alpha^2 E_s^k,
	\end{aligned}
\end{equation}
where we have used Lemma~\ref{sigma_leq1} in the first inequality, and $(a+b)^2 \leq (1+\varpi)a^2 + (1+\frac{1}{\varpi})b^2 $ for any $a,b\in\mathbb{R}$ and $\varpi > 0$ in the second inequality.

Second, we bound the tracking error $E_g^k$. By summing over $i\in[N]$ on both sides of~\eqref{variance} and noting that
\[
E_g^k=\sum\nolimits_{i}\mathbb{E}\!\left[ \| g_i^k - \nabla f_i(x_i^k) \|^2 \right],
\]
we can get
\begin{equation} \label{EgMidd}
    \begin{aligned}
        E_g^{k+1} \leq{} &  (1-p)\!\left(1+\frac{p}{2}\right)E_g^k + L^2C_u\sum\nolimits_i(u_i^k)^2 \\
        &+ 6d(1-p)L^2\,\mathbb{E}\!\left[\left\| x^{k+1} - x^k \right\|^2\right].
    \end{aligned}
\end{equation}
To bound the third term on the RHS of~\eqref{EgMidd}, we note that
\begin{equation} \label{xkkxkMid}
    \begin{aligned}
          \mathbb{E}\!\left[\left\| x^{k+1} - x^k \right\|^2\right]
          ={} &  \mathbb{E}\Big[\big\| (x^{k+1} - \mathbf{1}_N \otimes\bar{x}^{k+1}) - (x^k - \mathbf{1}_N \otimes\bar{x}^k) \\
          & + (\mathbf{1}_N \otimes\bar{x}^{k+1} - \mathbf{1}_N \otimes\bar{x}^k) \big\|^2\Big] \\
         \leq{} & 2E_x^{k+1} + 4E_x^k + 4N\ext{2}\mathbb{E}\!\left[\left\| \bar{x}^{k+1} - \bar{x}^k \right\|^2\right].
    \end{aligned}
\end{equation}
Here we bound $E_x^{k+1}$ differently as follows:
\begin{equation} \label{SimpEx}
    \begin{aligned}
        E_x^{k+1}
		={} &  \mathbb{E} \!\left[ \left\| (W \otimes I_d)\big[x^k \!-\! \mathbf{1}_N \otimes \bar{x}^k \!-\! \alpha (s^k \!-\! \mathbf{1}_N \otimes \bar{g}^k) \big] \right\|^2 \right] \\
        \leq{} & \sigma^2\,
        \mathbb{E}\!\left[ \left\| x^k \!-\! \mathbf{1}_N \otimes \bar{x}^k \!-\! \alpha (s^k \!-\! \mathbf{1}_N \otimes \bar{g}^k) \right\|^2 \right] \\
        \leq{} & \sigma^2\, \mathbb{E}\!\left[ 2\!\left(\left\|x^k \!-\! \mathbf{1}_N \otimes \bar{x}^k \right\|^2 \!+\! \left\| \alpha (s^k \!-\! \mathbf{1}_N \otimes \bar{g}^k) \right\|^2 \right)\right] \\
        \leq{} & 2E_x^k + 2\alpha^2 E_s^k.
    \end{aligned}
\end{equation}
Plugging~\eqref{SimpEx} into the inequality~\eqref{xkkxkMid}, we derive that 
\begin{equation} \label{xkkxkMidddd}
    \begin{aligned}
        \mathbb{E}\!\left[\left\| x^{k+1} \!-\! x^k \right\|^2\right] \leq 8E_x^k + 4\alpha^2 E_s^k + 4N \ext{2}\mathbb{E}\!\left[\left\| \bar{x}^{k+1} \!-\! \bar{x}^k \right\|^2\right]\!.
    \end{aligned}
\end{equation}
Now, we can combine~\eqref{xkkxkMidddd} and~\eqref{EgMidd} and get the desired bound on $E_g^{k+1}$ in Lemma~\ref{LMI_lemma}. 

Third, we bound the tracking error $E_s^k$. Note that
\begin{equation*}
	\begin{aligned}
		E_s^{k+1} ={} &  \mathbb{E} \!\left[ \left\| s^{k+1} - \mathbf{1}_N \otimes \bar{g}^{k+1} \right\|^2 \right] \\
		={} & \mathbb{E} \Big[  \big\| (W\otimes I_d) ( s^k - \mathbf{1}_N \otimes \bar{g}^k + g^{k+1} - g^k \\
		& \quad - \mathbf{1}_N \otimes \bar{g}^{k+1} + \mathbf{1}_N \otimes \bar{g}^k ) \big\|^2 \Big] .
	\end{aligned}
\end{equation*}
Since $\bar{s}^k=\bar{g}^k$, we may apply Lemma~\ref{sigma_leq1} to obtain
\begin{equation}\label{six_items}
\begin{aligned}
E_s^{k+1}\leq{} &
\sigma^2\,\mathbb{E} \Big[  \big\| s^k - \mathbf{1}_N \otimes \bar{g}^k + g^{k+1} - g^k \\
		& \quad - \mathbf{1}_N \otimes \bar{g}^{k+1} + \mathbf{1}_N \otimes \bar{g}^k \big\|^2 \Big] \\
\leq{} & \sigma^2
\mathbb{E} \Big[\big(  \big\| s^k - \mathbf{1}_N \otimes \bar{g}^k\big\| + \big\|g^{k+1} - g^k \\
		& \quad - \mathbf{1}_N \otimes (\bar{g}^{k+1} - \bar{g}^k) \big\|\big)^2 \Big].
\end{aligned}
\end{equation}
To bound the term $\big\|g^{k+1} - g^k - \mathbf{1}_N \otimes (\bar{g}^{k+1}-\bar{g}^k)\big\|$, note that
\begin{align}
		&   \| g^{k+1} - g^k - \mathbf{1}_N \otimes (\bar{g}^{k+1} - \bar{g}^k) \|^2  \nonumber\\*
		={} & \| g^{k+1} - g^k \|^2 + N\| \bar{g}^{k+1} - \bar{g}^k \|^2 \nonumber\\
        & - 2\sum\nolimits_{i=1}^N \langle g_i^{k+1} -  g_i^k, \bar{g}^{k+1} - \bar{g}^k \rangle \nonumber\\
		={} & \| g^{k+1} - g^k \|^2 - N\| \bar{g}^{k+1} - \bar{g}^k \|^2 \nonumber \\
		\leq{} & \| g^{k+1} - g^k \|^2.
  \label{g_k+1_k}
\end{align}
Combining~\eqref{g_k+1_k} with~\eqref{six_items}, we derive that
\begin{align}
		E_s^{k+1}
		\leq{} & \sigma^2 \,\mathbb{E}\!\left[  \left( \left\| s^k - \mathbf{1}_N \otimes \bar{g}^k \right\| +  \left\| g^{k+1} - g^k \right\| \right)^{\!2} \right]   \nonumber\\
        \leq{} &
        \sigma^2\left(1\!+\!\frac{1\!-\!\sigma^2}{3\sigma^2}\right)E_s^k
        +\sigma^2\left(1\!+\!\frac{3\sigma^2}{1\!-\!\sigma^2}\right)\mathbb{E}\!\left[ \left\| g^{k+1} - g^k \right\|^2 \right] \nonumber\\
		\leq{} & \frac{1+2\sigma^2}{3}E_s^k + \frac{3}{1-\sigma^2} \sum\nolimits_{i} \mathbb{E} \left[  \left\| g_i^{k+1} - g_i^k \right\|^2 \right]. \label{two_items}
\end{align}
Now we consider the second item in \eqref{two_items}:
\begin{equation} \label{g_k+1-g_k}
	\begin{aligned}
		\sum\nolimits_{i} \mathbb{E}\!\left[ \left\| g_i^{k+1} - g_i^k \right\|^2 \right]
		={} &  \mathbb{E} \Big[ \sum\nolimits_{i}\big\| \big(g_i^{k+1} - \nabla f_i(x_i^{k+1})\big) - \big(g_i^k - \nabla f_i(x_i^k)\big) \\
		& + \big(\nabla f_i(x_i^{k+1}) - \nabla f_i(x_i^k)\big) \big\|^2 \Big] \\
		\leq{} & 2E_g^{k+1} + 4E_g^k + 4L^2\,\mathbb{E} \!\left[ \left\| x^{k+1} - x^k \right\|^2 \right] , \\
        \leq{} &  6E_g^k + ( 12d(1-p) + 4 )L^2 \mathbb{E} [ \| x^{k+1} - x^k \|^2 ]  \\
        & + 2L^2C_u\sum\nolimits_i(u_i^k)^2,
	\end{aligned}
\end{equation}
where we have used Assumption~\ref{assumption_smooth_f^*} in the first inequality, and~\eqref{EgMidd} with $(1-p)(1+p/2)<1$ in the second inequality.

Plugging~\eqref{g_k+1-g_k} into~\eqref{two_items}, we will obtain
\begin{equation}
    \begin{aligned}
        E_s^{k+1} \leq{} & \frac{1\!+\!2\sigma^2}{3}E_s^k + \frac{12( 3d(1\!-\!p) + 1 )L^2}{1-\sigma^2} \mathbb{E}\!\left[\left\| x^{k+1} \!-\! x^k \right\|^2\right] \\
        &  + \frac{18}{1-\sigma^2}E_g^k + \frac{6L^2C_u\sum_i (u_i^k)^2}{1-\sigma^2}.
    \end{aligned}
\end{equation}
Using the inequality~\eqref{xkkxkMidddd}, we further derive that
\begin{equation} \label{boundbeforeEs}
    \begin{aligned}
        E_s^{k+1} \leq{} & \!\left[\frac{1\!+\!2\sigma^2}{3} \!+\! \frac{48( 3d(1\!-\!p) + 1 )\alpha^2L^2}{1\!-\!\sigma^2} \right]\!E_s^k \!+\! \frac{18}{1 \!-\! \sigma^2}E_g^k \\
        & + \frac{96( 3d(1-p) + 1 )L^2}{1-\sigma^2}E_x^k + \frac{6L^2C_u\sum_i (u_i^k)^2}{1-\sigma^2} \\
        & + \frac{48( 3d(1-p) + 1 )NL^2}{1-\sigma^2}\mathbb{E}\!\left[ \left\| \bar{x}^{k+1} - \bar{x}^k \right\|^2\right].
    \end{aligned}
\end{equation}

Using $\alpha L = c\sqrt{\frac{p}{d(1-p) + 1}}$ and $c\leq \frac{(1-\sigma^2)^2}{28^2} $, we derive that 
\begin{equation} \label{boundEs}
    \frac{48( 3d(1\!-\!p) + 1 )\alpha^2L^2}{1\!-\!\sigma^2} < \frac{1-\sigma^2}{3}.
\end{equation}
Combining~\eqref{boundEs} with~\eqref{boundbeforeEs}, we complete the proof.

\section{Proof of Lemma \ref{Sigma_inequality}}
\label{appendix:proof_lemma_Sigma_inequality}

Accurately determining and bounding the spectral radius or the spectral norm of the matrix $A$ is challenging. By introducing the auxiliary variable $E_c^k$, we can reduce the dimensionality of the system matrix $A$ from $\mathbb{R}^{3\times3}$ to $\mathbb{R}^{2\times2}$, making it more straightforward to analyze.

We first derive a bound for the variable $E_c^k$. By the definition of $E_c^k$, we see that
\begin{equation*} 
    \begin{aligned}
        E_c^{k+1}={} & E_x^{k+1} + \frac{18\alpha^2}{(1-\sigma^2)^2} E_s^{k+1} \\
        \leq{} & \left( \frac{1\!+\!2\sigma^2}{3} \!+\!  \frac{1728(3d(1\!-\!p)\!+\!1)\alpha^2L^2}{(1\!-\!\sigma^2)^3} \right)\!E_x^k  \!+\! \frac{324\alpha^2}{(1\!-\!\sigma^2)^3} E_g^k \\
        & + \!\left( \frac{1\!-\!\sigma^2}{6} \!+\! \frac{2\!+\!\sigma^2}{3} \right) \frac{18\alpha^2}{(1\!-\!\sigma^2)^2} E_s^k \!+\! \frac{108\alpha^2L^2C_u\sum\nolimits_i(u_i^k)^2}{(1-\sigma^2)^3} \\
        & + \frac{864( 3d(1-p) + 1 ) N\alpha^2L^2}{(1-\sigma^2)^3} \mathbb{E}\!\left[\left\| \bar{x}^{k+1} - \bar{x}^k \right\|^2\right],
    \end{aligned}
\end{equation*}
where we have used~\eqref{linear_matrix} in the inequality. Using $\alpha L = c\sqrt{\frac{p}{d(1-p) + 1}}$ and $c\leq \big(\frac{1-\sigma^2}{28}\big)^2 $, we derive that 
\[
\frac{1728( 3d(1\!-\!p) \!+\! 1 )\alpha^2L^2}{(1\!-\!\sigma^2)^3} < \frac{1-\sigma^2}{2}.
\]
Consequently, we have
\begin{equation}  \label{EcFinal}
    \begin{aligned}
        E_c^{k+1} \leq{} & \frac{5+\sigma^2}{6}E_c^k + \frac{324\alpha^2}{(1-\sigma^2)^3} E_g^k + \frac{108\alpha^2L^2C_u\sum_i (u_i^k)^2}{(1-\sigma^2)^3} \\
        & + \frac{864(3d(1-p) + 1) N\alpha^2 L^2}{(1-\sigma^2)^3} \mathbb{E}\!\left[ \left\| \bar{x}^{k+1} - \bar{x}^k \right\|^2\right].
    \end{aligned}
\end{equation}

We then derive a bound for the tracking error $E_g^k$ as follows:
\begin{equation} \label{EgEc}
    \begin{aligned}
        E_g^{k+1} \!\leq{} &  (1 - p)\!\left(1+\frac{p}{2}\right)E_g^k + 48 d(1\!-\!p)L^2E_x^k \\
        &+ \frac{4 d(1\!-\!p)L^2(1\!-\!\sigma^2)^2}{3}\frac{18\alpha^2}{(1\!-\!\sigma^2)^2} E_s^k \\
        & + 24 Nd(1\!-\!p)L^2\mathbb{E}\!\left[ \| \bar{x}^{k+1} \!-\! \bar{x}^k \|^2 \right] \!+\! L^2C_u\!\sum\nolimits_i\!(u_i^k)^2       \\ 
        \leq{} &  \!\!\left(1\!-\!\frac{p}{2}\right)\!E_g^k \!+\! 16(3d(1\!-\!p)\!+\!1)L^2E_c^k\!+\! L^2C_u\!\sum\nolimits_i\!(u_i^k)^2 \\
        &+ 8 N(3d(1\!-\!p)+1)L^2\mathbb{E}\!\left[ \| \bar{x}^{k+1} \!-\! \bar{x}^k \|^2 \right],
    \end{aligned}
\end{equation}
where we have used~\eqref{linear_matrix} in the first inequality, and the definition of $E_c^k$ as well as $(1-p)(1+p/2)\leq 1-p/2$ in the second inequality.

Now we are able to reformulate the inequality~\eqref{linear_matrix}. By combining inequality~\eqref{EcFinal} and~\eqref{EgEc} and take symmetric scaling, we derive that
\begin{equation} \label{symmetric_trans}
    \begin{aligned}
        w^{k+1} \leq C w^k + 
        \theta^k,
        \quad\text{where }\theta^k=
        \begin{bmatrix}
            \theta_1^k \\ \theta_2^k
        \end{bmatrix},
    \end{aligned}
\end{equation}
and
\begin{align*}
 w^k ={} & \left[ \frac{2L}{9\alpha}\sqrt{(3d(1\!-\!p)\!+\!1)(1\!-\!\sigma^2)^3} E_c^k ,\ \  E_g^k \right]^T \\
	    C ={} &  \! \begin{bmatrix}
            1 - \mfrac{1-\sigma^2}{6} & 72\alpha L \left[\mfrac{3d(1\!-\!p)\!+\!1}{(1-\sigma^2)^3}\right]^{\!1/2} \\[0.5em]  
             72\alpha L \left[\mfrac{3d(1\!-\!p)\!+\!1}{(1-\sigma^2)^3}\right]^{\!1/2}& 1-p/2
        \end{bmatrix}, \\
     \theta_1^k ={} & \frac{192(3d(1\!-\!p) \!+\! 1 )^{\frac{3}{2}} N\alpha L^3}{(1\!-\!\sigma^2)^{\frac{3}{2}}} \mathbb{E}\!\left[ \left\| \bar{x}^{k+1} \!-\! \bar{x}^k \right\|^2 \right] \\
       &+ \frac{24\alpha L^3 (3d(1-p)+1)^{\frac{1}{2}}C_u\sum_i(u_i^k)^2 }{(1-\sigma^2)^{\frac{3}{2}}} , \\ 
     \theta_2^k ={} & \ext{-2}8 N(3d(1\!-\!p)\!+\!1)L^2\mathbb{E}\ext{-5}\left[\ext{-2} \left\| \bar{x}^{k+1} \!-\! \bar{x}^k \right\|^2 \ext{-1}\right] \!+\! L^2C_u\!\sum\nolimits_i \!(u_i^k)^2.
\end{align*}

We denote 
\[
c_1 = \frac{1-\sigma^2}{6},~  c_2 = \frac{p}{2},~ c_3 =  72\alpha L \left[\frac{3d(1\!-\!p)\!+\!1}{(1-\sigma^2)^3}\right]^{\!1/2}.
\]
Using the property that the spectral norm equals the spectral radius for real symmetric matrices, we derive that
\begin{equation*} \label{norm_C}
\begin{aligned}
    	\| C \|_2 ={} & 1 - \frac{c_1+c_2}{2} + \frac{\sqrt{c_1^2-2c_1c_2 + c_2^2 + 4c_3^2}}{2} \\
    ={} & 1\!-\! \frac{c_1\!+\!c_2}{2}\!+\!\frac{1}{2} \sqrt{c_1^2+\frac{3\!+\!\sigma^2}{4}c_2^2 + \frac{1\!-\!\sigma^2}{4}c_2^2 - 2c_1c_2 + 4c_3^2}.
\end{aligned}
\end{equation*}  
Using $\alpha L = c\sqrt{\frac{p}{d(1-p) + 1}}$ and $c\leq \big(\frac{1-\sigma^2}{28}\big)^2 $, we derive that
\begin{align*}
& \frac{1\!-\!\sigma^2}{4}c_2^2 - 2c_1c_2 + 4c_3^2 \\
={} &
-\!(1\!-\!\sigma^2)p\!\left(\frac{1}{6}\!-\!\frac{p}{16}\right)
\!+\!\frac{20376c^2(3d(1-p)+1)p}{(d(1\!-\!p)\!+\!1)(1\!-\!\sigma^2)^3} \\
<{} &
-\frac{5(1-\sigma^2)p}{48} + \frac{62208c^2p}{(1-\sigma^2)^3} < 0.
\end{align*}
Consequently, we have
\[
\begin{aligned}
\| C \|_2 \leq{} 
& 1\!-\! \frac{c_1\!+\!c_2}{2}\!+\!\frac{1}{2} \sqrt{c_1^2+\frac{3\!+\!\sigma^2}{4}c_2^2} \\
\leq{} &
1\!-\! \frac{c_1\!+\!c_2}{2}\!+\!\frac{1}{2}\left(c_1+\sqrt{\frac{3\!+\!\sigma^2}{4}c_2^2}\right)= 1 - \chi p,
\end{aligned}
\]
where $\chi = \frac{1}{4} - \frac{1}{8}\sqrt{3+\sigma^2}$ and it is not hard to verify that $\chi \in (\frac{1-\sigma^2}{32}, \frac{1-\sigma^2}{29})$.

Now, from the definition of $E_f^k$ in Lemma~\ref{Sigma_inequality}, 
we take the $\ell_2$ norm on both sides of~\eqref{symmetric_trans} and get
\[
E_f^{k+1} 
        \leq\| C \|_2 E_f^k + \| \theta^k \|
\leq (1-\chi p) E_f^k + \|\theta^k\|.
\]
To bound $\|\theta^k\|$, using the condition on $\alpha$ and by some algebraic calculation, we can show that
\[
\begin{aligned}
\theta_1^k
<{} & \frac{4}{9}(3d(1\!-\!p)\!+\!1\!)NL^2\mathbb{E}\!\left[\ext{-2} \left\| \bar{x}^{k+1} \!-\! \bar{x}^k \right\|^2 \right]
\!+\!\frac{L^2 C_u\sum\nolimits_i(u_i^k)^2}{18} \\
={} & \frac{\theta_2^k}{18},
\end{aligned}
\]
which leads to
$\|\theta^k\|\leq{}\sqrt{1+\frac{1}{18^2}}|\theta_2^k|
\leq{} \frac{9}{8}|\theta_2^k|$ and
\begin{equation} \label{EfFinal}
    \begin{aligned}
        E_f^{k+1} 
        \leq{} & (1-\chi p) E_f^k + \frac{9}{8}L^2C_u\!\sum\nolimits_i \!(u_i^k)^2 \\
        & + 9 (3d(1\!-\!p)\!+\!1)NL^2\mathbb{E}\ext{-5}\left[\ext{-2} \left\| \bar{x}^{k+1} \!-\! \bar{x}^k \right\|^2 \right].
    \end{aligned}
\end{equation}
By induction on~\eqref{EfFinal}, we derive that for $k\geq 1$,
\begin{equation} \label{E_f_Appendix}
    \begin{aligned}
        E_f^k \leq{} & 9(3d(1\!-\!p) \!+\! 1 ) N L^2 \sum_{m=0}^{k-1}\! (1 \!-\! \chi p)^m \mathbb{E}[\| \bar{x}^{k-m} \!-\! \bar{x}^{k-m-1} \|^2 ]  \\ 
        & + (1 \!-\! \chi p)^k E_f^0 + \frac{9L^2C_u}{8} \sum_{m=0}^{k-1}(1 \!-\! \chi p)^m \!\sum\nolimits_i(u_i^{k-m-1})^2.
    \end{aligned}
\end{equation}
Taking sum over iteration $k$ on both sides of the inequality~\eqref{E_f_Appendix} and using~\eqref{sum_sum}, we obtain
\begin{equation} 
    \begin{aligned}
        \sum_{\tau=0}^{k} E_f^{\tau} \leq{} & 
        \frac{9( 3d(1\!-\!p) \!+\! 1 ) NL^2}{\chi p} \sum_{m=0}^{k-1} \mathbb{E}[\| \bar{x}^{m+1} - \bar{x}^m \|^2] \\
        & + \frac{1}{\chi p}E_f^0 + \frac{9L^2C_u}{ 8\chi p}\sum_{m=0}^{k-1} \sum\nolimits_i(u_i^m)^2.
    \end{aligned}
\end{equation}
The proof is now complete.

\section{Proof of Lemma \ref{barxkkbarxk}}
\label{appendix:proof_lemma_barxkkbarxk}

Based on~\eqref{xbar_kplus1}, we have
\begin{equation} \label{barxkkbargk}
    \begin{aligned}
        \mathbb{E}\!\left[ \left\| \bar{x}^{k+1} - \bar{x}^k \right\|^2 \right]={} & \alpha^2\mathbb{E}\!\left[\left\| \bar{g}^k \right\|^2\right] \\
        ={} &\alpha^2\mathbb{E}\!\left[\left\| \nabla f(\bar{x}^k)+\bar{g}^k -\nabla f(\bar{x}^k)\right\|^2\right]
        \\
        \leq{} & 2 \alpha^2 \mathbb{E}\!\left[ \big\| \nabla f(\bar{x}^k) \big\|^2 \right] + 2 \alpha^2 \mathbb{E}\!\left[ \left\| \nabla f(\bar{x}^k)-\bar{g}^k \right\|^2 \right].
    \end{aligned}
\end{equation}
Combining~\eqref{fbarTemp} with~\eqref{barxkkbargk}, we derive that
\begin{equation*}
    \begin{aligned}
        \mathbb{E}\!\left[ \left\| \bar{x}^{k+1} - \bar{x}^k \right\|^2 \right]
        \leq{} & 2 \alpha^2 \mathbb{E}\!\left[ \big\| \nabla f(\bar{x}^k) \big\|^2\right] +  \frac{4\alpha^2 L^2}{N}E_x^k + \frac{4\alpha^2}{N} E_g^k.
    \end{aligned}
\end{equation*}
Using the definitions of $E_f^k$ and $E_c^k$, we have
\[
E_x^k \leq E_c^k \leq 
\frac{9\alpha}{2L}
\left[\frac{1}{(3d(1\!-\!p)\!+\!1)(1\!-\!\sigma^2)^3}\right]^{\!\frac{1}{2}}E_f^k,\ 
E_g^k \leq E_f^k.
\]
Consequently, by using the condition on $\alpha$, we have
\begin{equation}
    \begin{aligned}
         \mathbb{E}\!\left[ \left\| \bar{x}^{k+1} - \bar{x}^k \right\|^2 \right]
         \leq{} & 2 \alpha^2 \mathbb{E} \!\left[ \big\| \nabla f(\bar{x}^k) \big\|^2 \right] +  \frac{4\alpha^2}{N} E_f^k  \\
         &  + \frac{4\alpha^2 L^2}{N} \cdot \frac{9\alpha}{2L}
\left[\frac{1}{(3d(1\!-\!p)\!+\!1)(1\!-\!\sigma^2)^3}\right]^{\!\frac{1}{2}} E_f^k \\
        \leq{} & 2 \alpha^2 \mathbb{E}\!\left[ \big\| \nabla f(\bar{x}^k) \big\|^2 \right] + \frac{5\alpha^2}{N} E_f^k.
    \end{aligned}
\end{equation}
We complete the proof.

\bibliographystyle{IEEEtran}       
\bibliography{autosam}    

\begin{thebibliography}{10}
\providecommand{\url}[1]{#1}
\csname url@samestyle\endcsname
\providecommand{\newblock}{\relax}
\providecommand{\bibinfo}[2]{#2}
\providecommand{\BIBentrySTDinterwordspacing}{\spaceskip=0pt\relax}
\providecommand{\BIBentryALTinterwordstretchfactor}{4}
\providecommand{\BIBentryALTinterwordspacing}{\spaceskip=\fontdimen2\font plus
\BIBentryALTinterwordstretchfactor\fontdimen3\font minus \fontdimen4\font\relax}
\providecommand{\BIBforeignlanguage}[2]{{%
\expandafter\ifx\csname l@#1\endcsname\relax
\typeout{** WARNING: IEEEtran.bst: No hyphenation pattern has been}%
\typeout{** loaded for the language `#1'. Using the pattern for}%
\typeout{** the default language instead.}%
\else
\language=\csname l@#1\endcsname
\fi
#2}}
\providecommand{\BIBdecl}{\relax}
\BIBdecl

\bibitem{8695072}
S.~Liang, L.~Y. Wang, and G.~Yin, ``Distributed smooth convex optimization with coupled constraints,'' \emph{IEEE Transactions on Automatic Control}, vol.~65, no.~1, pp. 347--353, 2020.

\bibitem{8006269}
X.~Zhang, A.~Papachristodoulou, and N.~Li, ``Distributed control for reaching optimal steady state in network systems: An optimization approach,'' \emph{IEEE Transactions on Automatic Control}, vol.~63, no.~3, pp. 864--871, 2018.

\bibitem{10093046}
J.~Liu, D.~W.~C. Ho, and L.~Li, ``A generic algorithm framework for distributed optimization over the time-varying network with communication delays,'' \emph{IEEE Transactions on Automatic Control}, vol.~69, no.~1, pp. 371--378, 2024.

\bibitem{nedic2009distributed}
A.~Nedi\'{c} and A.~Ozdaglar, ``Distributed subgradient methods for multi-agent optimization,'' \emph{IEEE Transactions on Automatic Control}, vol.~54, no.~1, pp. 48--61, 2009.

\bibitem{yuan2016convergence}
K.~Yuan, Q.~Ling, and W.~Yin, ``On the convergence of decentralized gradient descent,'' \emph{SIAM Journal on Optimization}, vol.~26, no.~3, pp. 1835--1854, 2016.

\bibitem{shi2015extra}
W.~Shi, Q.~Ling, G.~Wu, and W.~Yin, ``{EXTRA}: An exact first-order algorithm for decentralized consensus optimization,'' \emph{SIAM Journal on Optimization}, vol.~25, no.~2, pp. 944--966, 2015.

\bibitem{qu2017harnessing}
G.~Qu and N.~Li, ``Harnessing smoothness to accelerate distributed optimization,'' \emph{IEEE Transactions on Control of Network Systems}, vol.~5, no.~3, pp. 1245--1260, 2017.

\bibitem{nedic2017achieving}
A.~Nedi\'{c}, A.~Olshevsky, and W.~Shi, ``Achieving geometric convergence for distributed optimization over time-varying graphs,'' \emph{SIAM Journal on Optimization}, vol.~27, no.~4, pp. 2597--2633, 2017.

\bibitem{pu2021sharp}
S.~Pu, A.~Olshevsky, and I.~C. Paschalidis, ``A sharp estimate on the transient time of distributed stochastic gradient descent,'' \emph{IEEE Transactions on Automatic Control}, vol.~67, no.~11, pp. 5900--5915, 2021.

\bibitem{8786146}
A.~Reisizadeh, A.~Mokhtari, H.~Hassani, and R.~Pedarsani, ``An exact quantized decentralized gradient descent algorithm,'' \emph{IEEE Transactions on Signal Processing}, vol.~67, no.~19, pp. 4934--4947, 2019.

\bibitem{7963560}
A.~Nedić, A.~Olshevsky, W.~Shi, and C.~A. Uribe, ``Geometrically convergent distributed optimization with uncoordinated step-sizes,'' in \emph{2017 American Control Conference (ACC)}, 2017, pp. 3950--3955.

\bibitem{8988200}
S.~Pu, W.~Shi, J.~Xu, and A.~Nedić, ``Push–pull gradient methods for distributed optimization in networks,'' \emph{IEEE Transactions on Automatic Control}, vol.~66, no.~1, pp. 1--16, 2021.

\bibitem{nedic2020distributed}
A.~Nedi\'{c}, ``Distributed gradient methods for convex machine learning problems in networks: Distributed optimization,'' \emph{IEEE Signal Processing Magazine}, vol.~37, no.~3, pp. 92--101, 2020.

\bibitem{ren2021distributed}
X.~Ren, D.~Li, Y.~Xi, and H.~Shao, ``Distributed global optimization for a class of nonconvex optimization with coupled constraints,'' \emph{IEEE Transactions on Automatic Control}, vol.~67, no.~8, pp. 4322--4329, 2021.

\bibitem{carnevale2024nonconvex}
G.~Carnevale, N.~Mimmo, and G.~Notarstefano, ``Nonconvex distributed feedback optimization for aggregative cooperative robotics,'' \emph{Automatica}, vol. 167, p. 111767, 2024.

\bibitem{3295285}
X.~Lian, C.~Zhang, H.~Zhang, C.-J. Hsieh, W.~Zhang, and J.~Liu, ``Can decentralized algorithms outperform centralized algorithms? {A} case study for decentralized parallel stochastic gradient descent,'' in \emph{Advances in Neural Information Processing Systems}, vol.~30, 2017.

\bibitem{zeng2018nonconvex}
J.~Zeng and W.~Yin, ``On nonconvex decentralized gradient descent,'' \emph{IEEE Transactions on signal processing}, vol.~66, no.~11, pp. 2834--2848, 2018.

\bibitem{scutari2019distributed}
G.~Scutari and Y.~Sun, ``Distributed nonconvex constrained optimization over time-varying digraphs,'' \emph{Mathematical Programming}, vol. 176, pp. 497--544, 2019.

\bibitem{sun2019convergence}
Y.~Sun, G.~Scutari, and A.~Daneshmand, ``Distributed optimization based on gradient tracking revisited: Enhancing convergence rate via surrogation,'' \emph{SIAM Journal on Optimization}, vol.~32, no.~2, pp. 354--385, 2022.

\bibitem{tatarenko2017non}
T.~Tatarenko and B.~Touri, ``Non-convex distributed optimization,'' \emph{IEEE Transactions on Automatic Control}, vol.~62, no.~8, pp. 3744--3757, 2017.

\bibitem{sun2019distributed}
H.~Sun and M.~Hong, ``Distributed non-convex first-order optimization and information processing: Lower complexity bounds and rate optimal algorithms,'' \emph{IEEE Transactions on Signal Processing}, vol.~67, no.~22, pp. 5912--5928, 2019.

\bibitem{bai2024complexity}
Y.~Bai, Y.~Liu, and L.~Luo, ``On the complexity of finite-sum smooth optimization under the {Polyak--{\L}ojasiewicz} condition,'' in \emph{Proceedings of the 41st International Conference on Machine Learning}, 2024, pp. 2392--2417.

\bibitem{li2021surrogate}
Z.~Li, Z.~Dong, Z.~Liang, and Z.~Ding, ``Surrogate-based distributed optimisation for expensive black-box functions,'' \emph{Automatica}, vol. 125, p. 109407, 2021.

\bibitem{bubeck2012regret}
S.~Bubeck and N.~Cesa-Bianchi, ``Regret analysis of stochastic and nonstochastic multi-armed bandit problems,'' \emph{Foundations and Trends{\textregistered} in Machine Learning}, vol.~5, no.~1, pp. 1--122, 2012.

\bibitem{malladi2023fine}
S.~Malladi, T.~Gao, E.~Nichani, A.~Damian, J.~D. Lee, D.~Chen, and S.~Arora, ``Fine-tuning language models with just forward passes,'' in \emph{Advances in Neural Information Processing Systems}, vol.~36, 2023, pp. 53\,038--53\,075.

\bibitem{zhang2024boosting}
Y.~Zhang, Y.~Zhou, K.~Ji, Y.~Shen, and M.~M. Zavlanos, ``Boosting one-point derivative-free online optimization via residual feedback,'' \emph{IEEE Transactions on Automatic Control}, vol.~69, no.~9, pp. 6309--6316, 2024.

\bibitem{chen2025regression}
X.~Chen and Z.~Ren, ``Regression-based single-point zeroth-order optimization,'' \emph{arXiv preprint arXiv:2507.04223}, 2025.

\bibitem{nesterov2017random}
Y.~Nesterov and V.~Spokoiny, ``Random gradient-free minimization of convex functions,'' \emph{Foundations of Computational Mathematics}, vol.~17, no.~2, pp. 527--566, 2017.

\bibitem{duchi2015optimal}
J.~C. Duchi, M.~I. Jordan, M.~J. Wainwright, and A.~Wibisono, ``Optimal rates for zero-order convex optimization: The power of two function evaluations,'' \emph{IEEE Transactions on Information Theory}, vol.~61, no.~5, pp. 2788--2806, 2015.

\bibitem{mhanna2024zero}
E.~Mhanna and M.~Assaad, ``Zero-order one-point gradient estimate in consensus-based distributed stochastic optimization,'' \emph{Transactions on Machine Learning Research}, 2024.

\bibitem{sahu2018communication}
A.~K. Sahu, D.~Jakovetic, D.~Bajovic, and S.~Kar, ``Communication-efficient distributed strongly convex stochastic optimization: Non-asymptotic rates,'' \emph{arXiv preprint arXiv:1809.02920}, 2018.

\bibitem{hajinezhad2019zone}
D.~Hajinezhad, M.~Hong, and A.~Garcia, ``{ZONE}: Zeroth-order nonconvex multiagent optimization over networks,'' \emph{IEEE Transactions on Automatic Control}, vol.~64, no.~10, pp. 3995--4010, 2019.

\bibitem{yi2022zeroth}
X.~Yi, S.~Zhang, T.~Yang, and K.~H. Johansson, ``Zeroth-order algorithms for stochastic distributed nonconvex optimization,'' \emph{Automatica}, vol. 142, p. 110353, 2022.

\bibitem{lin2024decentralized}
Z.~Lin, J.~Xia, Q.~Deng, and L.~Luo, ``Decentralized gradient-free methods for stochastic non-smooth non-convex optimization,'' in \emph{Proceedings of the AAAI Conference on Artificial Intelligence}, vol.~38, no.~16, 2024, pp. 17\,477--17\,486.

\bibitem{sahinoglu2024online}
E.~Sahinoglu and S.~Shahrampour, ``An online optimization perspective on first-order and zero-order decentralized nonsmooth nonconvex stochastic optimization,'' in \emph{Proceedings of the 41st International Conference on Machine Learning}, 2024, pp. 43\,043--43\,059.

\bibitem{kiefer1952stochastic}
J.~Kiefer and J.~Wolfowitz, ``Stochastic estimation of the maximum of a regression function,'' \emph{The Annals of Mathematical Statistics}, pp. 462--466, 1952.

\bibitem{9199106}
Y.~Tang, J.~Zhang, and N.~Li, ``Distributed zero-order algorithms for nonconvex multiagent optimization,'' \emph{IEEE Transactions on Control of Network Systems}, vol.~8, no.~1, pp. 269--281, 2021.

\bibitem{guo2021machine}
Y.~Guo, D.~Coey, M.~Konutgan, W.~Li, C.~Schoener, and M.~Goldman, ``Machine learning for variance reduction in online experiments,'' in \emph{Advances in Neural Information Processing Systems}, vol.~34, 2021, pp. 8637--8648.

\bibitem{wang2013variance}
C.~Wang, X.~Chen, A.~J. Smola, and E.~P. Xing, ``Variance reduction for stochastic gradient optimization,'' in \emph{Advances in Neural Information Processing Systems}, vol.~26, 2013.

\bibitem{johnson2013accelerating}
R.~Johnson and T.~Zhang, ``Accelerating stochastic gradient descent using predictive variance reduction,'' in \emph{Advances in Neural Information Processing Systems}, vol.~26, 2013.

\bibitem{fang2018spider}
C.~Fang, C.~J. Li, Z.~Lin, and T.~Zhang, ``{SPIDER}: Near-optimal non-convex optimization via stochastic path-integrated differential estimator,'' \emph{Advances in Neural Information Processing Systems}, vol.~31, 2018, full version available at \url{https://arxiv.org/abs/1807.01695}.

\bibitem{li2021page}
Z.~Li, H.~Bao, X.~Zhang, and P.~Richt{\'a}rik, ``{PAGE}: A simple and optimal probabilistic gradient estimator for nonconvex optimization,'' in \emph{Proceedings of the 38th International Conference on Machine Learning}, 2021, pp. 6286--6295.

\bibitem{liu2018zeroth}
S.~Liu, B.~Kailkhura, P.-Y. Chen, P.~Ting, S.~Chang, and L.~Amini, ``Zeroth-order stochastic variance reduction for nonconvex optimization,'' in \emph{Advances in Neural Information Processing Systems}, vol.~31, 2018.

\bibitem{kazemi2024efficient}
E.~Kazemi and L.~Wang, ``Efficient zeroth-order proximal stochastic method for nonconvex nonsmooth black-box problems,'' \emph{Machine Learning}, vol. 113, no.~1, pp. 97--120, 2024.

\bibitem{xin2020variance}
R.~Xin, U.~A. Khan, and S.~Kar, ``Variance-reduced decentralized stochastic optimization with accelerated convergence,'' \emph{IEEE Transactions on Signal Processing}, vol.~68, pp. 6255--6271, 2020.

\bibitem{jiang2022distributed}
X.~Jiang, X.~Zeng, J.~Sun, and J.~Chen, ``Distributed stochastic gradient tracking algorithm with variance reduction for non-convex optimization,'' \emph{IEEE Transactions on Neural Networks and Learning Systems}, vol.~34, no.~9, pp. 5310--5321, 2022.

\bibitem{chen2023zeroth}
H.~Chen, J.~Chen, and K.~Wei, ``A zeroth-order variance-reduced method for decentralized stochastic non-convex optimization,'' \emph{arXiv preprint arXiv:2310.18883}, 2023.

\bibitem{xu2021distributed}
J.~Xu, Y.~Tian, Y.~Sun, and G.~Scutari, ``Distributed algorithms for composite optimization: Unified framework and convergence analysis,'' \emph{IEEE Transactions on Signal Processing}, vol.~69, pp. 3555--3570, 2021.

\bibitem{ji2019improved}
K.~Ji, Z.~Wang, Y.~Zhou, and Y.~Liang, ``Improved zeroth-order variance reduced algorithms and analysis for nonconvex optimization,'' in \emph{Proceedings of the 36th International Conference on Machine Learning}, 2019, pp. 3100--3109.

\bibitem{yi2021linear}
X.~Yi, S.~Zhang, T.~Yang, T.~Chai, and K.~H. Johansson, ``Linear convergence of first-and zeroth-order primal--dual algorithms for distributed nonconvex optimization,'' \emph{IEEE Transactions on Automatic Control}, vol.~67, no.~8, pp. 4194--4201, 2021.

\bibitem{publications_tang}
H.~Mu, Y.~Tang, and Z.~Li, ``Variance-reduced gradient estimator for nonconvex zeroth-order distributed optimization,'' in \emph{2025 American Control Conference (ACC)}, 2025.

\bibitem{polyak1964gradient}
B.~T. Polyak, ``Gradient methods for solving equations and inequalities,'' \emph{USSR Computational Mathematics and Mathematical Physics}, vol.~4, no.~6, pp. 17--32, 1964.

\bibitem{lojasiewicz1963topological}
S.~{{\L}ojasiewicz}, ``A topological property of real analytic subsets,'' \emph{Coll. du CNRS, Les {\'e}quations aux d{\'e}riv{\'e}es partielles}, vol. 117, no. 87-89, p.~2, 1963.

\bibitem{nemirovskij1983problem}
A.~S. Nemirovskij and D.~B. Yudin, \emph{Problem Complexity and Method Efficiency in Optimization}.\hskip 1em plus 0.5em minus 0.4em\relax Wiley-Interscience, 1983.

\bibitem{1070486}
A.~D. Flaxman, A.~T. Kalai, and H.~B. McMahan, ``Online convex optimization in the bandit setting: gradient descent without a gradient,'' in \emph{Proceedings of the Sixteenth Annual ACM-SIAM Symposium on Discrete Algorithms}, 2005, pp. 385--394.

\bibitem{krizhevsky2009learning}
A.~Krizhevsky, ``Learning multiple layers of features from tiny images,'' University of Toronto, Tech. Rep., 2009.

\bibitem{xiao2005scheme}
L.~Xiao, S.~Boyd, and S.~Lall, ``A scheme for robust distributed sensor fusion based on average consensus,'' in \emph{IPSN 2005. Fourth International Symposium on Information Processing in Sensor Networks}, 2005, pp. 63--70.

\bibitem{nesterov2018lectures}
Y.~Nesterov, \emph{Lectures on Convex Optimization}.\hskip 1em plus 0.5em minus 0.4em\relax Springer Cham, 2018.

\end{thebibliography}

\end{document}